\documentclass{neutral}

\newcommand{\PH}[1]{\mathrm{PH}^{#1}(\Omega)}

\title{Sensitivity analysis for 3D Maxwell's equations and its use in the resolution of an inverse medium problem at fixed frequency}

\addauthor{Marion Darbas}{marion.darbas@u-picardie.fr}{LAMFA CNRS UMR 7352 - Université de Picardie Jules Verne, 33 rue Saint-Leu, 80039 Amiens cedex 1}
\addauthor{Jérémy Heleine}{jeremy.heleine@u-picardie.fr}{LAMFA CNRS UMR 7352 - Université de Picardie Jules Verne, 33 rue Saint-Leu, 80039 Amiens cedex 1}
\addauthor{Stephanie Lohrengel}{stephanie.lohrengel@univ-reims.fr}{LMR CNRS FRE2011 - Université de Reims Champagne-Ardenne, Moulin de la Housse, 51687 Reims cedex 2}


\submitted{11 October 2018}
\accepted{9 February 2019}
\published{12 March 2019}
\pubjournal{Inverse Problems in Science and Engineering}
\puburl{https://doi.org/10.1080/17415977.2019.1588896}

\begin{document}

	\maketitle

	\begin{abstract}
		This paper deals with the reconstruction of small-amplitude perturbations in the electric properties (permittivity and conductivity) of a medium from boundary measurements of the electric field at a fixed frequency. The underlying model are the three-dimensional time-harmonic Maxwell equations in the electric field. Sensitivity analysis with respect to the parameters is performed, and explicit relations between the boundary measurements and the characteristics of the perturbations are found from an appropriate integral equation and extensive numerical simulations in 3D. The resulting non-iterative algorithm allows to retrieve efficiently the center and volume of the perturbations in various situations from the simple sphere to a realistic model of the human head.
	\end{abstract}

	\displaykeywords

	\section{Introduction}

	The study of dielectric properties of biological tissues or materials is of great interest in medical or industrial applications. The dielectric behavior of a tissue and its interaction with electromagnetic fields are able to describe and provide information about its characteristics and composition. This information can be used to develop new noninvasive modalities in many practical applications of electric fields in agriculture, bioengineering and medical diagnosis. Dielectric properties of biological tissues are frequency-dependent or dispersive, and experimental investigation has shown that they vary with respect to low or high frequencies of the applied electric field or current (e.g.~\cite{Alanen99, Gabriel09}). Among the different imaging modalities based on electric fields, we may cite Electrical Impedance Tomography (EIT) which operates at frequencies between \num{10} and \SI{100}{\kilo\hertz} and Microwave Imaging between \SI{300}{\mega\hertz} and \SI{300}{\giga\hertz}. The principle of EIT is to provide the electrical permittivity and conductivity inside a body from simultaneous measurements of electrical currents and potentials at the boundary. With regard to medical applications, microwave imaging is studied with the aim of detecting and monitoring cerebrovascular accidents (or strokes). Indeed, strokes result in variations of the dielectric properties of the affected tissues, and experimental research has found that the contrast of dielectric parameters between the abnormal and normal cerebral tissues can be imaged within the microwave spectrum (at frequencies of the order of \SI{1}{\giga\hertz}). New devices based on these properties are currently designed and studied~\cite{Semenov14, Tournier17}. Microwave breast imaging offers also a promising alternative method to mammography~\cite{Kwon16}. Compared to other medical imaging technologies such as Magnetic Resonance Imaging (MRI) and Computarized Tomography (CT-scan), EIT and microwave imaging have a low resolution due to the ill-posed nature of the image reconstruction problem. There is a lot of interest (low cost device, harmless procedure, \ldots) in finding ways to improve their resolution and these modalities are areas of active research.

	From a mathematical point of view, one has to deal with the theoretical and numerical study of an inverse medium problem. The goal is to retrieve the complex refractive index of a medium, namely the electric permittivity (real part) and conductivity (imaginary part) from boundary measurements at a fixed frequency. This inverse problem is severely ill-posed (e.g~\cite{RomanovKabanikhin}). Indeed, coefficients of elliptic problems (like the conductivity equation) in a bounded domain are uniquely determined by the entire (scalar or vector) Dirichlet-to-Neumann map on the whole boundary of the domain which is in general not available in practical applications. The fundamental example of parameter reconstruction is Calderón's inverse conductivity problem~\cite{Calderon}. The theoretical and numerical study of the EIT inverse problem has also been extensively addressed in the last two decades. We refer for instance to~\cite{Borcea,Ammari09,SeoWoo,Ammari17} and references therein.

	In this paper, we focus on the inverse medium problem associated with the time-harmonic 3D Maxwell equations formulated in the electric field with a possible application in microwave imaging. For uniqueness and stability results for the Maxwell system from total or partial data, we refer the reader for instance to the works of Ola, Païvärinta and Somersalo \cite{Ola93}, Caro \textit{et al} \cite{Caro10, CaroZhou}, Kenig, Salo and Uhlman \cite{KenigSaloUhlmann} and references therein. In practice, only partial information on the (vector) Dirichlet-to-Neumann map is available. The challenging issues are thus to provide numerical methods for reconstructing the dielectric properties of a medium from a finite number of boundary measurements of the electric field. A classical way consists in formulating the inverse problem as the minimization of a cost function representing the difference between the measured and predicted fields. To solve the minimization problem, a gradient-based algorithm is currently used. For instance, Beilina \textit{et al}~(e.g. \cite{Beilina11,Beilina16}) have developed an adaptive finite element method based on a posteriori estimates for the simultaneous reconstruction of the real-valued electric permittivity and magnetic permeability functions of the 3D Maxwell's system. De Buhan and Darbas~\cite{deBuhanDarbas} have combined the quasi-Newton BFGS method and an iterative process (called the Adaptive Eigenspace Inversion) for determining the complex dielectric permittivity of a medium with 2D numerical validations. Another way to express the inverse medium problem is to search small anomalies in the electric parameters on a known background. We can cite the significant results of Ammari \textit{et al}~(e.g. \cite{AmmariKang,AVV01,AmmariVolkov}) who have derived small-volume expansions of the electromagnetic field, the volume of the imperfections being the asymptotic parameter. This yields constructive numerical methods for the localization of small-volume electromagnetic defects from measurements on a part of the boundary (see e.g. \cite{AschMefire} for 3D numerical results). This asymptotic approach has also been combined with an exact controllability method for retrieving small-amplitude perturbations in the permeability of a medium~\cite{Ammari03}. In this case, the time-dependent Maxwell equations and dynamic boundary measurements are considered. The performance of the reconstruction method in 2D has been addressed in~\cite{DarbasLohrengel}. In the present work, our aim is to detect and identify small-amplitude perturbations in the dielectric parameters of a medium from (time-independent) boundary measurements of the electric field. In this sense, our work falls into the previous class of approaches. We propose to investigate the problem from a different point of view. Our reconstruction method is based on explicit relations between sensitivity (with respect to the physical parameters) and characteristics of the imperfections. Sensitivity is the derivation of a given quantity (cost functional, physical field, \ldots) with respect to the parameters. Sensitivity gives an interesting tool for understanding the impact of local changes in interior parameters on the observed boundary measurements at the surface of the studied object. For instance, sensitivity information has been recently used to answer concrete clinic questions in EEG (electroencephalography) for neonates~\cite{MBE18} or to design resolution-based discretizations of the conductivity space in EIT~\cite{WinklerRieder}. In practical applications, the parameters are in general discretized, for instance by P0 or P1 finite elements. This leads to a Jacobian matrix which can be used to find the critical points of some least-square functional \cite{Dorn_etal_02,Dehghani_etal_09}. Another approach that is widely studied, is the topological sensitivity (or shape sensitivity) with respect to the shape of a perturbation in the parameters (see e.g. \cite{Amstutz06,Ren_etal_18} for EIT). The topological sensitivity is also used for the detection and shape identification of scatterers (e.g \cite{Bonnet_etal_13} and \cite{LeLouer1, LeLouer2} for acoustics and electromagnetism respectively). Data are in this case measurements of the far-field pattern of the scattered field.

	In the present paper, the aim is to investigate the impact of small-amplitude perturbations in the dielectric parameters of a medium on boundary measurements of the electric field. The novelty lies in proposing a rigourous sensitivity analysis of the electric field with respect to the variations of the electric permittivity and conductivity, noticing that the perturbation in the measurements is proportionnal to sensitivity for small-amplitudes of the parameters. We address both theoretical and numerical aspects. The sensitivity analysis is the first step for developing a new non- iterative inversion algorithm that allows to determine the location and volume of small-amplitude anomalies from boundary measurements of the perturbed electric field. To our knowledge, it's the first time that such a sensitivity analysis with respect to parameters (and not to the shape) is realized for solving an inverse electromagnetic medium problem.

	The remainder of the paper is organized as follows. In \autoref{section2}, we present the forward problem under consideration and define the functional setting. In \autoref{sec:sensitivity}, we propose a theoretical sensitivity analysis of the electric field with respect to the electric permittivity and/or conductivity of a medium which is illustrated by numerical simulations. \autoref{sensitivity_sec} is devoted to the sensitivity analysis in the case of a constant background which allows to write the sensitivity boundary data as solution of an integral equation. In \autoref{sec:inverse}, we explain how to use the previous results for solving an inverse medium problem. The localization procedure is described in \autoref{sec:algo}, and various three-dimensional numerical results are reported to illustrate the method. Finally, we give some conclusions and perspectives in the last section.

	\section{The forward problem}
	\label{section2}
	\subsection{Time-harmonic Maxwell's equations}

	Let $\Omega$ denote a bounded and simply connected domain in $\R^3$ of Lipschitz boundary $\Gamma = \partial\Omega$. The unit outward normal to $\Omega$ is denoted by $\bfn$. We introduce the vector spaces
	\begin{align*}
		H(\curl) &= \set*{\bfu \in L^2(\Omega)^3}{\curl\bfu \in L^2(\Omega)^3}, \\
		Y(\Gamma) &= \set*{\bff \in H^{-1/2}(\Gamma)^3}{\Exists{\bfu \in H(\curl)} \bfu \times \bfn = \bff}.
	\end{align*}

	We are interested in time-harmonic Maxwell's equations with Neumann boundary condition:
	\begin{equation}
		\label{eq:maxwell}
		\left\{
		\begin{array}{rcl@{\hspace{4\tabcolsep}}l}
			\curl\curl\bfE - k^2\kappa\bfE &=& \bfF, & \stext[r]{in} \Omega, \\
			\curl\bfE \times \bfn &=& \bfg, & \stext[r]{on} \Gamma.
		\end{array}
		\right.
		\tag{$\mcM$}
	\end{equation}
	Here, $\bfE$ denotes the electric field intensity in $\Omega$, and the fields $\bfF$ and $\bfg$ are given source terms in, respectively, $L^2(\Omega)^3$ and $Y(\Gamma)$. The number $k \coloneqq \omega\sqrt{\mu_0\eps_0}$ is the wavenumber with $\omega$ the wave angular frequency, $\mu_0$ and $\eps_0$ respectively the magnetic permeability and electric permittivity in vacuum. We assume that the magnetic permeability in $\Omega$ is equal to $\mu_0$. Let $\eps$ and $\sigma$ denote, respectively, the electric permittivity and conductivity in $\Omega$. The refractive index $\kappa$ of the medium in $\Omega$ is defined by:
	\[\kappa(\bfx) = \frac{1}{\eps_0}\left(\eps(\bfx) + i\frac{\sigma(\bfx)}{\omega}\right), \quad \bfx \in \Omega.\]

	We assume that $\Omega$ is decomposed into $P$ connected Lipschitz subdomains, denoted by $\Omega_p$ for $1 \leq p \leq P$, such that
	\[\bar{\Omega} = \bigcup_{p=1}^P \bar{\Omega}_p \quad \text{and} \quad \Omega_p \cap \Omega_q = \emptyset, \quad \sforall p \neq q.\]
	Moreover, the following assumptions are made on the parameter $\kappa$
	\begin{equation}
		\label{eq:hyp_kappa}
		\left\{
		\begin{array}{l}
			\Forall{p \in \zinterval{1}{P}} \restriction{\kappa}{\Omega_p} \in H^3(\Omega_p), \\
			\Exists{\alpha_R > 0} \real{\kappa} \geq \alpha_R \stext{in} \Omega, \\
			\Exists{\alpha_I > 0} \Forall{p \in \zinterval{1}{P}} \stext[r]{either} \imag*{\restriction{\kappa}{\Omega_p}} \geq \alpha_I \stext{or} \imag*{\restriction{\kappa}{\Omega_p}} = 0, \\
			\imag{\kappa} \not\equiv 0 \stext{in} \Omega.
		\end{array}
		\right.
		\tag{$\mathcal{H}_\kappa$}
	\end{equation}

	The variational formulation of \eqref{eq:maxwell} is
	\begin{equation}
		\label{eq:maxwell_var}
		\left\{
		\begin{array}{l}
			\stext[r]{Find} \bfE \in H(\curl) \stext[l]{such that} \\
			\dotprod{\curl\bfE}{\curl\bfphi}{} - k^2\dotprod{\kappa\bfE}{\bfphi}{} = \dotprod{\bfF}{\bfphi}{} + \duality{\bfg}{\bfphi_T}{\Gamma}, \sforall \bfphi \in H(\curl),
		\end{array}
		\right.
		\tag{$\mcM_v$}
	\end{equation}
	where $\dotprod{}{}{}$ denotes the dot-product in $L^2(\Omega)^3$ and $\duality{}{}{\Gamma}$ denotes the dual product from $Y(\Gamma)$ to $Y(\Gamma)'$. Here, ${(\cdot)}_T$ denotes the extension to $H(\curl)$ of the map:
	\[\bfv \mapsto \bfn \times (\restriction{\bfv}{\Gamma} \times \bfn),\]
	classically defined on $\mcC(\bar{\Omega})^3$ (see \cite[Theorem~3.31]{Monk03}).

	\begin{theorem}
		\label{thm:sol_maxwell}
		Under the above assumptions, the problem~\eqref{eq:maxwell_var} admits a unique solution $\bfE$ in $H(\curl)$ depending continuously on $\bfF$ and $\bfg$.
	\end{theorem}

	\textit{Sketch of the proof:} The proof is adapted from \cite{Monk03}. The main ingredients are a judicious Helmholtz decomposition of $H(\curl)$ and
	the Fredholm alternative.

	\subsection{Regularity of the variational solution}

	In a setting where the subdomains $\Omega_p$ are only Lipschitz, the solution of \eqref{eq:maxwell_var} may have poor regularity. Indeed, singularities are likely to occur at the corners and edges of the subdomains, and the solution $\bfE$ does not belong, in general, to $H^1(\Omega_p)^3$ (see \cite{CoDaNi99}). In the present paper and with regard to the (biomedical) applications that we have in mind, we do not deal with these questions of singularities. Therefore, we assume from now on that $\Omega$ as well as all subdomains $\Omega_p$ are at least of class $\mcC^{1,1}$. This assumption allows to prove regularity results for the solution of \eqref{eq:maxwell_var} and the sensitivity equation that will be presented hereafter.

	According to the partition of $\Omega$ into $P$ subdomains $(\Omega_p)_p$, we introduce the spaces of piecewise smooth functions: for $s>0$, let
	\[\PH{s} = \set{v \in L^2(\Omega)}{\restriction{v}{\Omega_p} \in H^s(\Omega_p) \sforall p \in \zinterval{1}{P}}.\]
	We adopt the notation $\PH{s}^3$ to denote spaces of piecewise smooth vector fields. We further introduce the classical space
	\begin{align*}
		H(\div) &= \set*{\bfu \in L^2(\Omega)^3}{\div\bfu \in L^2(\Omega)} \\
		\shortintertext{as well as the trace space}
		H^{s}(\div_\Gamma) &= \set*{\bff \in H^{s}_t(\Gamma)}{\div_\Gamma\bff \in H^{s}(\Gamma)}
	\end{align*}
	where $\div_\Gamma$ denotes the surface divergence operator defined on the subspace $H^{s}_t(\Gamma)$ of $ H^s(\Gamma)^3$ of tangential fields $\bff$. A rigorous definition of $\div_\Gamma$ can be found in \cite{CK98,Monk03}. The following regularity result can be deduced from \cite{CoDaNi99}:

	\begin{theorem}
		\label{thm:reg_maxwell}
		Let $\Omega$ as well as all subdomains $\Omega_p$ be of class $\mcC^{1,1}$. Assume that the source term $\bfF$ belongs to $H(\div)$ and satisfies $\bfF \cdot \bfn \in H^{1/2}(\Gamma)$. Assume further that $\bfg \in H^{1/2}(\div_\Gamma)$. Let $\kappa$ be a piecewise constant function with respect to the partition of $\Omega$ that satisfies the assumptions of \autoref{thm:sol_maxwell}. Then, the solution of \eqref{eq:maxwell_var} belongs to $\PH{1}^3$ and satisfies
		\begin{subequations}
			\begin{empheq}[left = \empheqlbrace]{align}
				\label{eq:div1}
				-k^2\div(\kappa\bfE) &= \div\bfF \stext{in} \Omega, \\
				\label{eq:div2}
				k^2\kappa\bfE \cdot \bfn &= -\bfF \cdot \bfn + \div_\Gamma\bfg \stext{on} \Gamma.
			\end{empheq}
		\end{subequations}
	\end{theorem}

	\begin{proof}
		Let $\bfE\in H(\curl)$ be the solution of \eqref{eq:maxwell_var}. Since $\bfE$ satisfies \eqref{eq:maxwell} in the distributional sense, we get \eqref{eq:div1} where the right hand side belongs to $L^2(\Omega)$. We then deduce \eqref{eq:div2} from Green's formula and \eqref{eq:maxwell_var}. Now, let $p \in \faktor{H^1(\Omega)}{\R}$ be the unique solution of the Neumann problem
		\[
			\left\{
			\begin{array}{rcl@{\hspace{4\tabcolsep}}l}
				-\div(\kappa\nabla p) &=& \div\bfF, & \stext[r]{in} \Omega, \\
				\kappa\partial_n p &=& -\bfF \cdot \bfn + \div_\Gamma\bfg, & \stext[r]{on} \Gamma.
			\end{array}
			\right.
		\]
		According to the assumptions on the data $\bfF$ and $\bfg$ and due to the regularity of $\Omega$ and its subdomains, the scalar potential $p$ belongs to $\PH{2}$. Then, let $\bfE_0 = \bfE - \dfrac{1}{k^2} \nabla p$. $\bfE_0$ obviously belongs to $H(\curl)$ and is divergence free in the sense that $\div(\kappa\bfE_0) = 0$. It also satifies the homogeneous boundary condition $\kappa\bfE_0 \cdot \bfn = 0$ on $\Gamma$. Therefore, we can apply \cite[Theorem~3.5]{CoDaNi99}, and deduce that the field $\bfE_0$ admits a decomposition $\bfE_0 = \bfE_{0,R} + \nabla p_0$ where $\bfE_{0,R} \in \PH{1}^3$ and $p_0 \in \faktor{H^1(\Omega)}{\R}$ is the unique solution of a Neumann problem with right hand side in $L^2(\Omega)$ and homogeneous boundary condition. Again, $p_0$ belongs to $\PH{2}$ in the present setting of regular subdomains. This shows that $\bfE_0$ belongs to $\PH{1}^3$ and implies $\bfE \in \PH{1}^3$ due to the regularity of the scalar potential $p$.
	\end{proof}

	In the case of regular data and a constant parameter $\kappa$, a stronger regularity result can be obtained for $\bfE$:
	\begin{theorem}
		\label{thm:reg_maxwell2}
		Let $\Omega$ be of class $\mcC^{2,1}$. Let $\kappa$ be a constant such that $\real{\kappa} \geq 0$ and $\imag{\kappa} \geq 0$. Assume that $\bfF \in H^1(\Omega)^3$ with $\div\bfF \in H^1(\Omega)$ and $\bfF \cdot \bfn \in H^{3/2}(\Gamma)$. Assume further that $\bfg \in H^{3/2}(\div_\Gamma)$. Then, $\bfE \in H^2(\Omega)^3$.
	\end{theorem}

	\begin{proof}
		The proof is based on a result from \cite{AmBeDaGi98}: if $\Omega$ is of class $\mcC^{m,1}$ for $m \in \N$, the spaces
		\[\set*{\bfv \in L^2(\Omega)^3}{\curl\bfv \in H^{m-1}(\Omega)^3; \div\bfv \in H^{m-1}(\Omega); \bfv \times \bfn \in H^{m-1/2}(\Gamma)^3}\]
		and
		\[\set*{\bfv \in L^2(\Omega)^3}{\curl\bfv \in H^{m-1}(\Omega)^3; \div\bfv \in H^{m-1}(\Omega); \bfv \cdot \bfn \in H^{m-1/2}(\Gamma)}\]
		are both continuously imbedded in $H^m(\Omega)^3$.

		Now, let $\bfw = \curl\bfE$. According to \eqref{eq:maxwell} and the regularity result of \autoref{thm:reg_maxwell}, we have $\curl\bfw = k^2\kappa\bfE + F \in H^1(\Omega)^3$ as well as $\div\bfw = 0\ \mbox{in}\ \Omega$ and $\bfw \times \bfn = \curl\bfE \times \bfn = \bfg \in H^{3/2}(\Gamma)$. Thus, $\curl\bfE = \bfw \in H^2(\Omega)^3$. Furthermore, the regularity assumptions on $\bfF$ and $\bfg$ ensure that $\div\bfE \in H^1(\Omega)$ and $\bfE \cdot \bfn \in H^{3/2}(\Gamma)$. Therefore, $\bfE \in H^2(\Omega)^3$.
	\end{proof}

	\section{Sensitivity analysis with respect to a perturbation of electric parameters}
	\label{sec:sensitivity}

	Sensitivity analysis determines how the solution of a problem varies when a slight perturbation is induced in some of its physical parameters. Here, we are interested in the sensitivity analysis of the electric field with respect to the electrical permittivity and/or conductivity. Mathematically, it may be described rigorously by the Gâteaux derivative (see for example \cite{BorggaardNunes11}).

	\begin{definition}
		Let $F\colon X \to Y$ be an application between two Banach spaces $X$ and $Y$. Let $U \subset X$ be an open set. The Gâteaux derivative of $F$ at $\tau \in U$ in the direction $\varrho \in X$ is defined as
		\[D_{\varrho} F(\tau) = \lim_{h \to 0} \frac{F(\tau + h\varrho) - F(\tau)}{h}\]
		if the limit exists. If it exists for any direction $\varrho \in X$ and if the application $\varrho \mapsto D_{\varrho} F(\tau)$ is linear and continuous from $X$ to $Y$, then we say that $F$ is Gâteaux differentiable at $\tau$.
	\end{definition}

	\subsection{Sensitivity equation}

	We define the space of parameters
	\[\mcP = \set*{(\eps,\sigma) \in L^\infty(\Omega)^2}{\kappa \in \PH{3}}\]
	which is a Banach space, equipped with the norm
	\[\norm{(\eps,\sigma)}{\mcP} = \max(\norm{\eps}{L^\infty},\norm{\sigma}{L^\infty}), \quad \forall (\eps,\sigma) \in \mcP.\]
	We define the open set of admissible parameters
	\[\mcP_\text{adm} = \set{(\eps,\sigma) \in \mcP}{\eps_\text{min} < \eps < \eps_\text{max} \stext{and} \sigma_\text{min} < \sigma < \sigma_\text{max} \stext{in} \Omega}\]
	where $0 < \eps_\text{min} < \eps_\text{max}$ and $0 < \sigma_\text{min} < \sigma_\text{max}$ are real constants. From \autoref{thm:sol_maxwell}, we deduce that, for any $\tau = (\eps,\sigma) \in \mcP_\text{adm}$, the problem~\eqref{eq:maxwell_var} admits a unique solution, denoted by $\bfE(\cdot,\tau)$.

	\begin{theorem}
		\label{thm:sensitivity}
		Let $\tau \in \mcP_\text{adm}$. Let $h_0 > 0$ be such that $\tau + h\varrho \in \mcP_\text{adm}$ for any $h \in \interval{-h_0}{h_0}$ and $\varrho = (\varrho_{\eps},\varrho_{\sigma}) \in \mcP$. Then $\bfE(\cdot,\tau)$ is Gâteaux differentiable at $\tau$ in the direction $\varrho$. Moreover, its derivative $D_{\varrho}\bfE(\cdot,\tau)$ is the unique solution of the following variational problem
		\begin{equation}
			\label{eq:sensitivity}
			\left\{
			\begin{array}{l}
				\stext[r]{Find} \bfE^1 \in H(\curl) \stext[l]{such that} \\
				\dotprod{\curl\bfE^1}{\curl\bfphi}{} - k^2\dotprod{\kappa\bfE^1}{\bfphi}{} = \dfrac{k^2}{\eps_0}\dotprod*{\left(\varrho_{\eps} + i\dfrac{\varrho_{\sigma}}{\omega}\right)\bfE}{\bfphi}{}, \forall \bfphi \in H(\curl).
			\end{array}
			\right.
			\tag{$\mcS$}
		\end{equation}
	\end{theorem}

	To simplify the writing of the proof of this result, we introduce the following notation, for any couple of positive (or null) reals $a$ and $b$:
	\[a \lesssim b \iff (\Exists{ C > 0} a \leq C b)\]
	where $C$ is a constant independent of $a$ and $b$.

	\begin{proof}
		Let $h \in \interval{-h_0}{h_0} \setminus \collection{0}$ and $\varrho = (\varrho_{\eps},\varrho_{\sigma}) \in \mcP$. Let $\bfE_h = \bfE(\cdot,\tau + h\varrho)$.

		The field $\bfE$ is the unique solution of
		\begin{equation}
			\label{eq:sensitivity_sol_maxwell}
			\left\{
			\begin{aligned}
			&	\stext[r]{Find} \bfE \in H(\curl) \stext[l]{such that} \\
			&	\dotprod{\curl\bfE}{\curl\bfphi}{} - \frac{k^2}{\eps_0}\dotprod*{\left(\eps + i\frac{\sigma}{\omega}\right)\bfE}{\bfphi}{} = \dotprod{\bfF}{\bfphi}{} + \duality{\bfg}{\bfphi_T}{\Gamma}, \quad \forall \bfphi \in H(\curl),
			\end{aligned}
			\right.
		\end{equation}
		whereas the field $\bfE_h$ is the unique solution of
		\begin{equation}
			\label{eq:sensitivity_sol_maxwell_h}
			\left\{
			\begin{aligned}
			&	\stext[r]{Find} \bfE_h \in H(\curl) \stext[l]{such that} \\
			&	\begin{multlined}[c][0.8\textwidth]
					\dotprod{\curl\bfE_h}{\curl\bfphi}{} - \frac{k^2}{\eps_0}\dotprod*{\left((\eps + h\varrho_{\eps}) + i\frac{\sigma + h\varrho_{\sigma}}{\omega}\right)\bfE_h}{\bfphi}{} \\
					= \dotprod{\bfF}{\bfphi}{} + \duality{\bfg}{\bfphi_T}{\Gamma}, \quad \forall \bfphi \in H(\curl).
				\end{multlined}
			\end{aligned}
			\right.
		\end{equation}
		Let $\bfphi \in H(\curl)$. We compute the difference between \eqref{eq:sensitivity_sol_maxwell_h} and \eqref{eq:sensitivity_sol_maxwell} and we divide by $h$ to find
		\begin{equation}
			\label{eq:sensitivity_difference}
			\dotprod{\curl\bfE^1_h}{\curl\bfphi}{} - k^2\dotprod{\kappa\bfE^1_h}{\bfphi}{} = \frac{k^2}{\eps_0}\dotprod*{\left(\varrho_{\eps} + i\frac{\varrho_{\sigma}}{\omega}\right)\bfE_h}{\bfphi}{},
		\end{equation}
		where $\bfE^1_h = \frac{\bfE_h - \bfE}{h}$.

		We now compute the difference between \eqref{eq:sensitivity_difference} and \eqref{eq:sensitivity} to obtain
		\begin{equation}
			\label{eq:sensitivity_limit_0}
			\dotprod{\curl(\bfE^1_h - \bfE^1)}{\curl\bfphi}{} - k^2\dotprod{\kappa(\bfE^1_h - \bfE^1)}{\bfphi}{} = \frac{k^2}{\eps_0}\dotprod*{\left(\varrho_{\eps} + i\frac{\varrho_{\sigma}}{\omega}\right)(\bfE_h - \bfE)}{\bfphi}{}.
		\end{equation}
		We note $\tilde{\bfE}_h = \bfE^1_h - \bfE^1$ and $\tilde{\bfF}_h = \dfrac{k^2}{\eps_0}\left(\varrho_{\eps} + i\dfrac{\varrho_{\sigma}}{\omega}\right)(\bfE_h - \bfE)$. We get:
		\begin{equation}
			\label{eq:sensitivity_limit}
			\dotprod{\curl\tilde{\bfE}_h}{\curl\bfphi}{} - k^2\dotprod{\kappa\tilde{\bfE}_h}{\bfphi}{} = \dotprod{\tilde{\bfF}_h}{\bfphi}{}.
		\end{equation}

		As $\varrho_{\eps}$ and $\varrho_{\sigma}$ are in $L^\infty(\Omega)$, we have $\tilde{\bfF}_h \in L^2(\Omega)^3$. Then \eqref{eq:sensitivity_limit} can be seen as the variational formulation of Maxwell's equations with a homogeneous Neumann boundary condition and the source term $\tilde{\bfF}_h$. From \autoref{thm:sol_maxwell}, we deduce that $\tilde{\bfE}_h \in H(\curl)$ is the unique field satisfying \eqref{eq:sensitivity_limit} for all $\bfphi \in H(\curl)$. Moreover, we know that
		\[\norm{\tilde{\bfE}_h}{H(\curl)} \lesssim \norm{\tilde{\bfF}_h}{0} \lesssim \norm{\bfE_h - \bfE}{0} \lesssim \norm{\bfE_h - \bfE}{H(\curl)}.\]
		We now use the definition of $\bfE^1_h$ to get
		\[|h|^{-1}\norm{\bfE_h - \bfE}{H(\curl)} = \norm{\bfE^1_h}{H(\curl)}.\]
		Furthermore, $\bfE^1_h$ satisfies \eqref{eq:sensitivity_difference} for all $\bfphi \in H(\curl)$ and we have
		\[\norm{\bfE^1_h}{H(\curl)} \lesssim \norm{\bfE_h}{0} \lesssim \norm{\bfE_h}{H(\curl)} \lesssim \norm{\bfF}{0} + \norm{\bfg}{H^{-1/2}}\]
		since $\bfE_h$ is the unique solution of \eqref{eq:sensitivity_sol_maxwell_h}. Combining these inequalities, we get
		\[\norm{\tilde{\bfE}_h}{H(\curl)} \lesssim |h|.\]
		Thus $\bfE^1_h$ converges to $\bfE^1$ in $H(\curl)$.

		\medskip

		In order to prove the linearity of the application $\varrho \mapsto D_\varrho\bfE(\cdot,\tau)$, let $\varrho = \lambda \varrho_1 + \varrho_2$ with $\lambda \in \C$ and $\varrho_j = (\varrho_{j,1},\varrho_{j,2}) \in \mcP, \forall j \in \collection{1,2}$. For $j \in \collection{1,2}$, we set $\bfE^1_j \coloneqq D_{\varrho_j}\bfE(\cdot,\tau)$. Thus $\bfE^1_j$ solves
		\[\dotprod{\curl\bfE^1_j}{\curl\bfphi}{} - k^2\dotprod{\kappa\bfE^1_j}{\bfphi}{} = \frac{k^2}{\eps_0}\dotprod*{\left(\varrho_{j,1} + i\frac{\varrho_{j,2}}{\omega}\right)\bfE}{\bfphi}{}, \; \forall \bfphi \in H(\curl).\]
		Let $\bfE^1 = \lambda\bfE^1_1 + \bfE^1_2$. By linearity, we have
		\[\dotprod{\curl\bfE^1}{\curl\bfphi}{} - k^2\dotprod{\kappa\bfE^1}{\bfphi}{} = \frac{k^2}{\eps_0}\dotprod*{\left(\varrho_{\eps} + i\frac{\varrho_{\sigma}}{\omega}\right)\bfE}{\bfphi}{}, \; \forall \bfphi \in H(\curl),\]
		where $\varrho = (\varrho_\eps, \varrho_\sigma) = (\lambda \varrho_{1,1} + \varrho_{2,1},\lambda \varrho_{1,2} + \varrho_{2,2})$. Then $\bfE^1$ is solution of the problem satisfied by $D_{\varrho}\bfE(\cdot,\tau)$. From the uniqueness of the solution, we deduce that
		\[D_{\lambda \varrho_1 + \varrho_2}\bfE(\cdot,\tau) = \lambda D_{\varrho_1}\bfE(\cdot,\tau) + D_{\varrho_2}\bfE(\cdot,\tau).\]

		We obtain that $D_{\varrho} \bfE(\cdot,\tau)$ is solution of \eqref{eq:sensitivity}. Moreover we have
		\[\norm{D_\varrho\bfE(\cdot,\tau)}{H(\curl)} \lesssim \norm*{\left(\varrho_{\eps} + i \dfrac{\varrho_{\sigma}}{\omega} \right)\bfE}{0} \lesssim \norm{\varrho}{\mcP}.\]
		Thus, the application $\varrho \mapsto D_\varrho\bfE(\cdot,\tau)$ is linear and continuous from $\mcP$ to $H(\curl)$.
	\end{proof}

	\subsection{Regularity of the solution to the sensitivity equation}

	The derivative $\bfE^1 = D_\varrho\bfE(\cdot,\tau)$ of $\bfE$ with respect to the parameter $\tau = (\varepsilon,\sigma)$ in the direction $\varrho = (\varrho_{\eps},\varrho_{\sigma})$ is solution of the following boundary value problem
	\begin{equation}
		\label{eq:sens-edp}
		\left\{
		\begin{array}{rcl@{\hspace{4\tabcolsep}}l}
			\curl\curl\bfE^1 - k^2\kappa\bfE^1 &=& k^2\chi\bfE, & \stext[r]{in} \Omega, \\
			\curl\bfE^1\times\bfn &=& 0, & \stext[r]{on} \Gamma
		\end{array}
		\right.
	\end{equation}
	where $\chi = \frac{1}{\varepsilon_0}\left(\varrho_{\eps} + i\frac{\varrho_{\sigma}}{\omega}\right)$.

	\begin{theorem}
		\label{thm:sens-reg}
		Let $\chi \in W^{1,\infty}(\Omega)$. Under the assumptions of \autoref{thm:reg_maxwell}, the solution of \eqref{eq:sensitivity} belongs to $\PH{1}$ and satisfies
		\begin{subequations}
			\begin{empheq}[left = \empheqlbrace]{align}
				\label{eq:sens-div1}
				-\div(\kappa\bfE^1) &= \div(\chi\bfE) \stext{in} \Omega, \\
				\label{eq:sens-div2}
				\kappa\bfE^1 \cdot \bfn &= -\chi\bfE \cdot \bfn \stext{on} \Gamma.
			\end{empheq}
		\end{subequations}
	\end{theorem}

	\begin{proof}
		Let $\bfE^1 \in H(\curl)$ be the solution of \eqref{eq:sensitivity}. The regularity assumption on $\chi$ implies that $\div(\chi\bfE)$ belongs to $L^2(\Omega)$, and \eqref{eq:sens-div1} follows immediately from \eqref{eq:sens-edp}. The second identity \eqref{eq:sens-div2} can be obtained as in \autoref{thm:reg_maxwell}. Since $\bfE$ belongs to $\PH{1}^3$, its normal trace on $\Gamma$ is an element of $H^{1/2}(\Gamma)$. The same arguments as in \autoref{thm:reg_maxwell} then yield $\bfE^1 \in \PH{1}^3$.
	\end{proof}

	As for the solution of \eqref{eq:maxwell_var}, we get more regularity in the case of a constant function $\kappa$.
	\begin{theorem}
		\label{thm:sens-reg2}
		Let $\chi \in W^{2,\infty}(\Omega)$ and $\kappa$ a constant. Under the assumptions of \autoref{thm:reg_maxwell2}, the solution of \eqref{eq:sensitivity} belongs to $H^2(\Omega)^3$.
	\end{theorem}

	\begin{proof}
		Under the given assumptions, the solution $\bfE$ of \eqref{eq:maxwell_var} belongs to $H^2(\Omega)^3$ according to \autoref{thm:reg_maxwell2}. Together with the regularity of $\chi$, we thus infer from \eqref{eq:sens-div1} and \eqref{eq:sens-div2} that $\div\bfE^1 \in H^1(\Omega)$ and $\bfE^1 \cdot \bfn \in H^{3/2}(\Gamma)$. As in the proof of \autoref{thm:reg_maxwell2}, $\curl\bfE$ can be shown to belong to $H^2(\Omega)^3$. The regularity result follows from \cite{AmBeDaGi98}.
	\end{proof}

	\subsection{Some properties of the sensitivity}

	In this section, we prove some properties of the sensitivity that are directly linked to the linearity of the Gâteaux derivative.

	\begin{proposition}
		\label{prop:partial_derivatives}
		Let $\tau \in \mcP_\text{adm}$. Let $\varrho = (\varrho_{\eps},\varrho_{\sigma}) \in \mcP$. Then we have
		\[D_\varrho\bfE(\cdot,\tau) = D_{(\varrho_{\eps},0)}\bfE(\cdot,\tau) + D_{(0,\varrho_{\sigma})}\bfE(\cdot,\tau).\]
	\end{proposition}

	\begin{proposition}
		\label{prop:same_direction}
		Let $\tau \in \mcP_\text{adm}$. Let $\varrho \in \PH{3}$. We set $\bfE^1_\eps \coloneqq D_{(\varrho,0)}\bfE(\cdot,\tau)$ and $\bfE^1_\sigma \coloneqq D_{(0,\varrho)}\bfE(\cdot,\tau)$. Then
		\[\bfE^1_\eps = -i\omega\bfE^1_\sigma.\]
	\end{proposition}

	\begin{proof}
		Let $\bfE^1 = \bfE^1_\eps + i\omega\bfE^1_\sigma$. The result will be proved if we show that $\bfE^1 = 0$. To this end, let $\bfphi \in H(\curl)$. We have
		\begin{equation}
			\label{eq:same_direction_E1}
			\begin{multlined}[c][0.9\textwidth]
				\dotprod{\curl\bfE^1}{\curl\bfphi}{} - k^2\dotprod{\kappa\bfE^1}{\bfphi}{} = \dotprod{\curl\bfE^1_\eps}{\curl\bfphi}{} - k^2\dotprod{\kappa\bfE^1_\eps}{\bfphi}{} \\
				+ i\omega\left(\dotprod{\curl\bfE^1_\sigma}{\curl\bfphi}{} - k^2\dotprod{\kappa\bfE^1_\sigma}{\bfphi}{}\right).
			\end{multlined}
		\end{equation}
		We then apply \autoref{thm:sensitivity} to $\bfE^1_\eps$ and $\bfE^1_\sigma$ to find
		\begin{equation}
			\label{eq:same_direction_E1eps}
			\dotprod{\curl\bfE^1_\eps}{\curl\bfphi}{} - k^2\dotprod{\kappa\bfE^1_\eps}{\bfphi}{} = \frac{k^2}{\eps_0}\dotprod{\varrho\bfE}{\bfphi}{}
		\end{equation}
		and
		\begin{equation}
			\label{eq:same_direction_E1sigma}
			\dotprod{\curl\bfE^1_\sigma}{\curl\bfphi}{} - k^2\dotprod{\kappa\bfE^1_\sigma}{\bfphi}{} = i\frac{k^2}{\eps_0\omega}\dotprod{\varrho\bfE}{\bfphi}{}.
		\end{equation}
		We now inject \eqref{eq:same_direction_E1eps} and \eqref{eq:same_direction_E1sigma} in \eqref{eq:same_direction_E1} to find
		\[\dotprod{\curl\bfE^1}{\curl\bfphi}{} - k^2\dotprod{\kappa\bfE^1}{\bfphi}{} = \left(\frac{k^2}{\eps_0} - \omega\frac{k^2}{\eps_0\omega}\right)\dotprod{\varrho\bfE}{\bfphi}{} = 0.\]
		Then $\bfE^1$ is solution of \eqref{eq:maxwell_var} with $\bfF = 0$ and $\bfg = 0$. By the uniqueness of the solution, we find that $\bfE^1 = 0$.
	\end{proof}

	\begin{remark}
		For large frequencies $\omega$, \autoref{prop:same_direction} thus implies that the derivatives of the electric field with respect to the parameters $\eps$ and $\sigma$ are not of the same order whenever the directions in which the derivatives are taken have comparable norms of order $\mcO(1)$. This statement suggests to study sensitivity with respect to the permittivity in a direction of order $1/\omega$. We refer to \autoref{fig:sensitivity_xy} for an illustration.
	\end{remark}

	We are interested in studying how the location of a perturbation affects the electric field. Therefore, we focus on derivatives in the direction of characteristic functions of the perturbations' supports. The numerical results of the \autoref{sec:numerical_results} show that in this case the sensitivity is localized and illustrate the following proposition.

	\begin{proposition}
		\label{prop:disjoint}
		Let $(P_j)_{1 \leq j \leq N}$ be a collection of $N \in \N^*$ subsets of $\Omega$ such that
		\[P_{j_1} \cap P_{j_2} = \emptyset, \quad \forall j_1 \neq j_2.\]
		For all $j \in \zinterval{1}{N}$, we denote by $\varrho_j$ the indicator function of $P_j$. Let $\varrho$ be the indicator function of $\bigcup_{j=1}^N P_j$. Then we have
		\[D_{(\varrho,0)}\bfE(\cdot,\tau) = \sum_{j=1}^N D_{(\varrho_j,0)}\bfE(\cdot,\tau) \quad \text{and} \quad D_{(0,\varrho)}\bfE(\cdot,\tau) = \sum_{j=1}^N D_{(0,\varrho_j)}\bfE(\cdot,\tau),\]
		for any $\tau \in \mcP_\text{adm}$.
	\end{proposition}

	\subsection{Numerical results and comments }
	\label{sec:numerical_results}

	We implemented the numerical solver for 3D Maxwell's equations with FreeFem++ (see~\cite{Hecht12}). Our test domain $\Omega$ is the unit ball of $\R^3$. We consider a tetrahedral mesh $\mcT_h$. For any $T \in \mcT_h$, let $h_T$ be its diameter. Then $h = \max_{T \in \mcT_h} h_T$ is the mesh parameter of $\mcT_h$. For any $h$, we denote by $N_e$ the number of edges. Edge finite elements of order 1 (see~\cite{Monk03,Nedelec86}) are used to approximate the respective solutions of the problem~\eqref{eq:maxwell_var} and of the sensitivity equation~\eqref{eq:sensitivity}.

	We consider that $\Omega$ is filled with a homogeneous medium of constant electrical permittivity $\eps = \SI{1e-8}{\farad\per\meter}$ and conductivity $\sigma = \SI{0.33}{\siemens\per\meter}$ at the fixed frequency $\omega = \SI{1e6}{\hertz}$. The mesh characteristics are $h = 0.12$ and $N_e = \num{167402}$. The sensitivity $\bfE^1$ of the electric field in a given direction $\varrho = (\varrho_{\eps}, \varrho_{\sigma})$ is computed as the solution of equation~\eqref{eq:sensitivity}.

	First, we compare the modulus of the sensitivity with respect to a perturbation either of the conductivity or the permittivity (see \autoref{fig:sensitivity_position}, left and right). This perturbation is modeled by a sphere $B = B_\alpha(\bfx_0)$ of radius $\alpha = 0.1$, centered at $\bfx_0 = (-0.8,0,0)$. The respective directions are $\varrho = (\textbf{1}_{B}/\omega,0)$ for the permittivity and $\varrho = (0, \textbf{1}_{B})$ for the conductivity. This result illustrates \autoref{prop:same_direction} which implies $|\bfE^1_\eps| = |\bfE^1_\sigma|$. In the sequel, we consider a perturbation of the conductivity only. In the bottom of \autoref{fig:sensitivity_position}, the perturbation is placed at a different position. The simulation indicates how the position of the inhomogeneity affects sensitivity. In particular, it shows that the sensitivity is localized to a surface area close the inhomogeneity.

	\begin{figure}[p]
		\centering
		\includegraphics[width=0.45\textwidth]{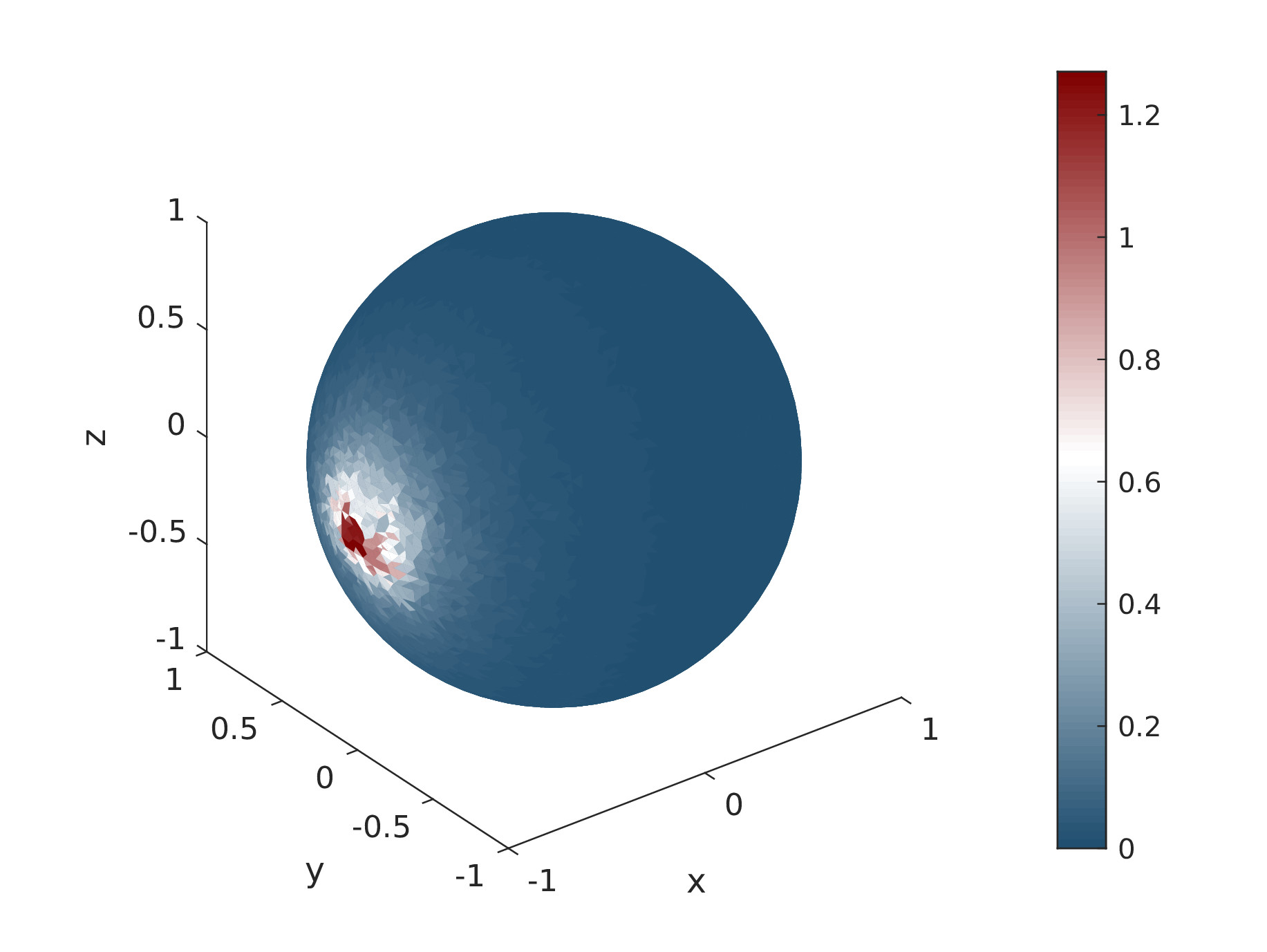}
		\hfill
		\includegraphics[width=0.45\textwidth]{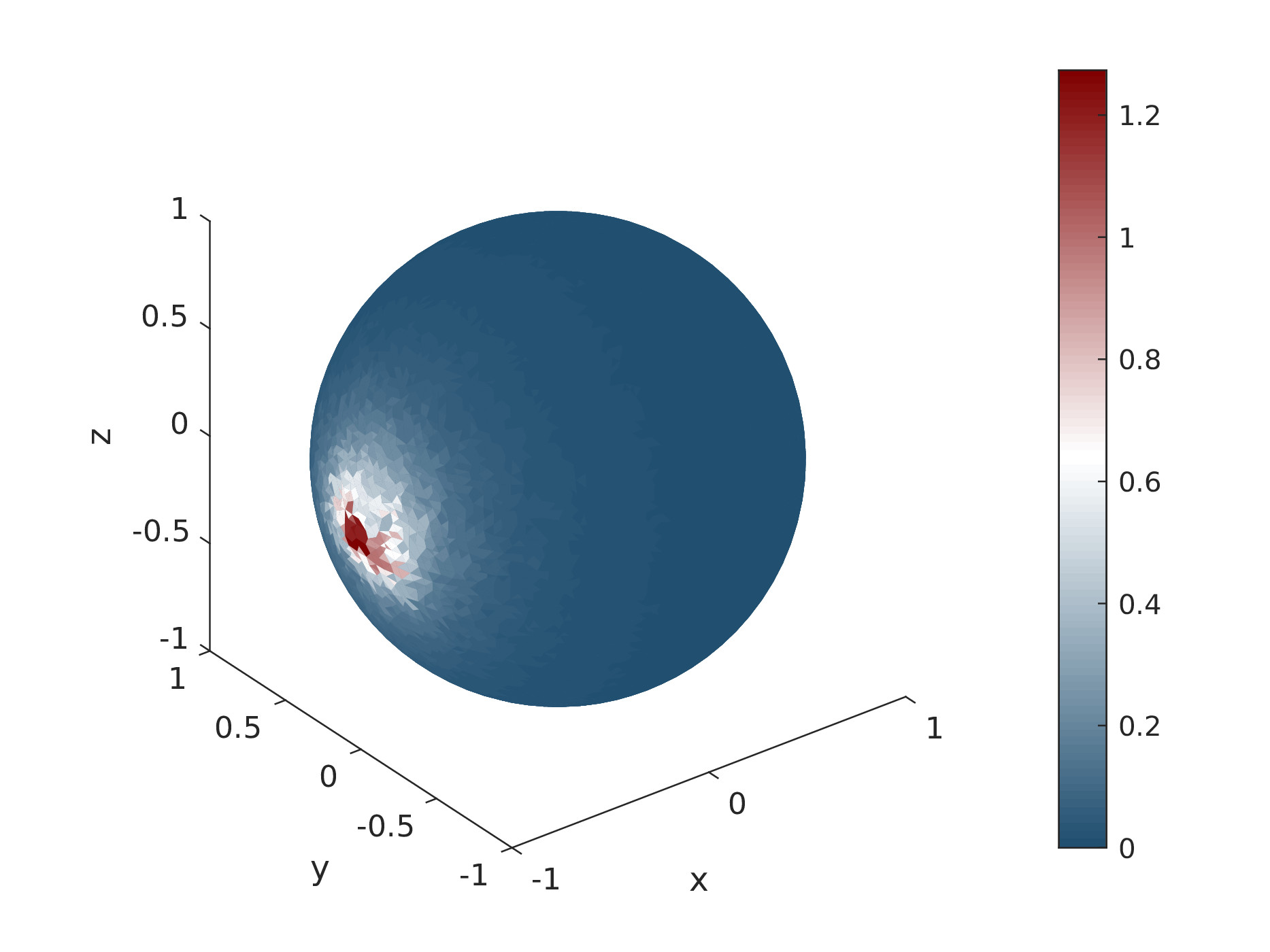}
		\begin{minipage}[c]{0.45\textwidth}
			\includegraphics[width=\textwidth]{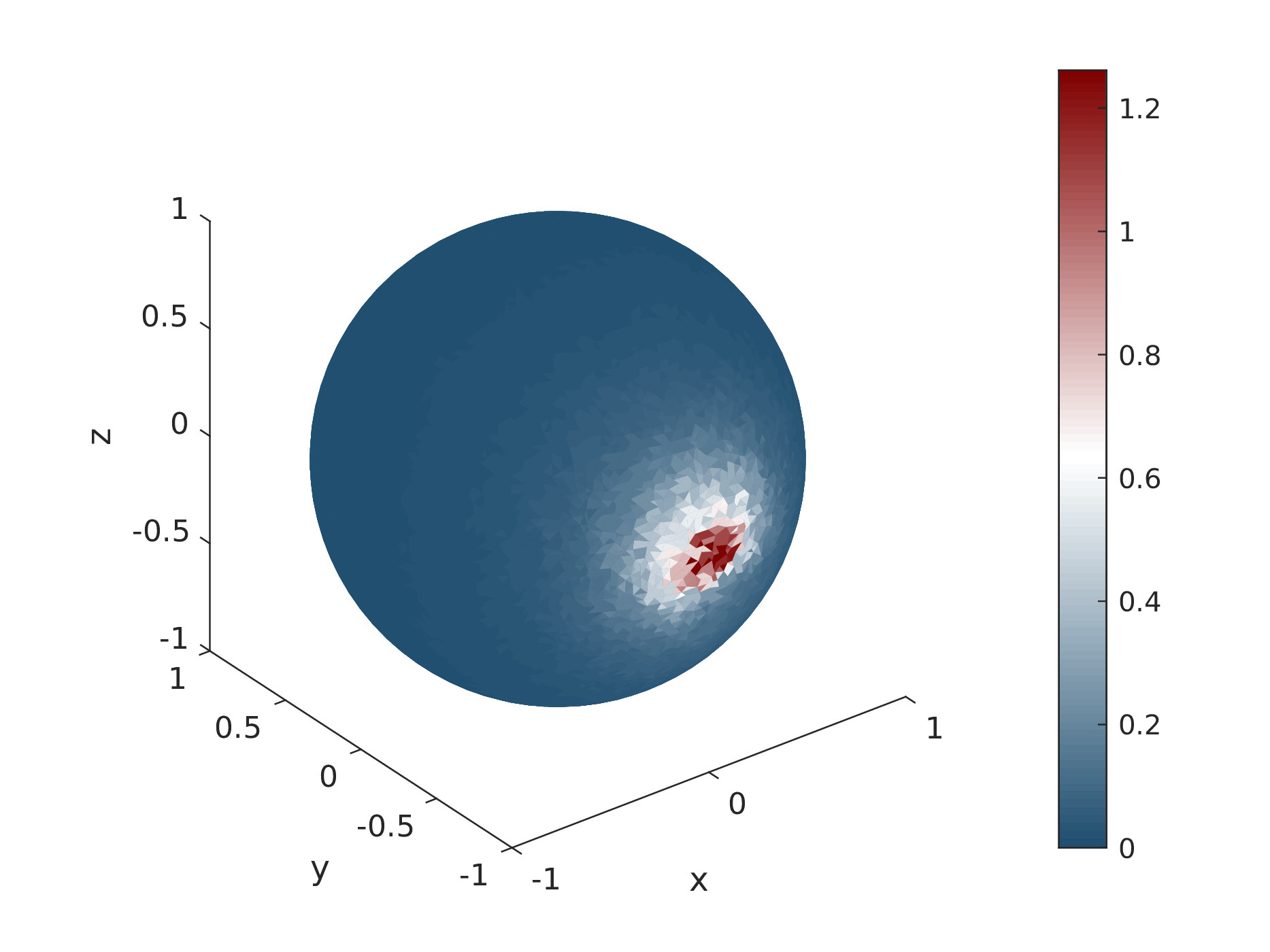}
		\end{minipage}
		\hfill
		\begin{minipage}[c]{0.45\textwidth}
			\caption{Sensitivity of the electric field on the surface with respect to the parameter (top left: inhomogeneity in the conductivity, top right: inhomogeneity in the permittivity) and the position of the inhomogeneity (top: centered at $\bfx_0 = (-0.8,0,0)$ on the x-axis, bottom: centered at $\bfx_0 = (0,-0.8,0)$ on the y-axis).}
			\label{fig:sensitivity_position}
		\end{minipage}
	\end{figure}

	In \autoref{fig:sensitivity_deep}, we report the sensitivity corresponding to a spherical inhomogeneity centered at $\bfx_0 = (-0.55,0,0)$ for different volumes. Compared to the perturbation at $\bfx_0 = (-0.8,0,0)$ (see \autoref{fig:sensitivity_position}, left), we observe that a deeper inhomogeneity leads to more spreaded surfacic perturbations. Moreover, increasing the inhomogeneity's size does not change the shape, but increases the amplitude of the sensitivity (see \autoref{fig:sensitivity_deep}, right).

	\begin{figure}[p]
		\centering
		\includegraphics[width=0.45\textwidth]{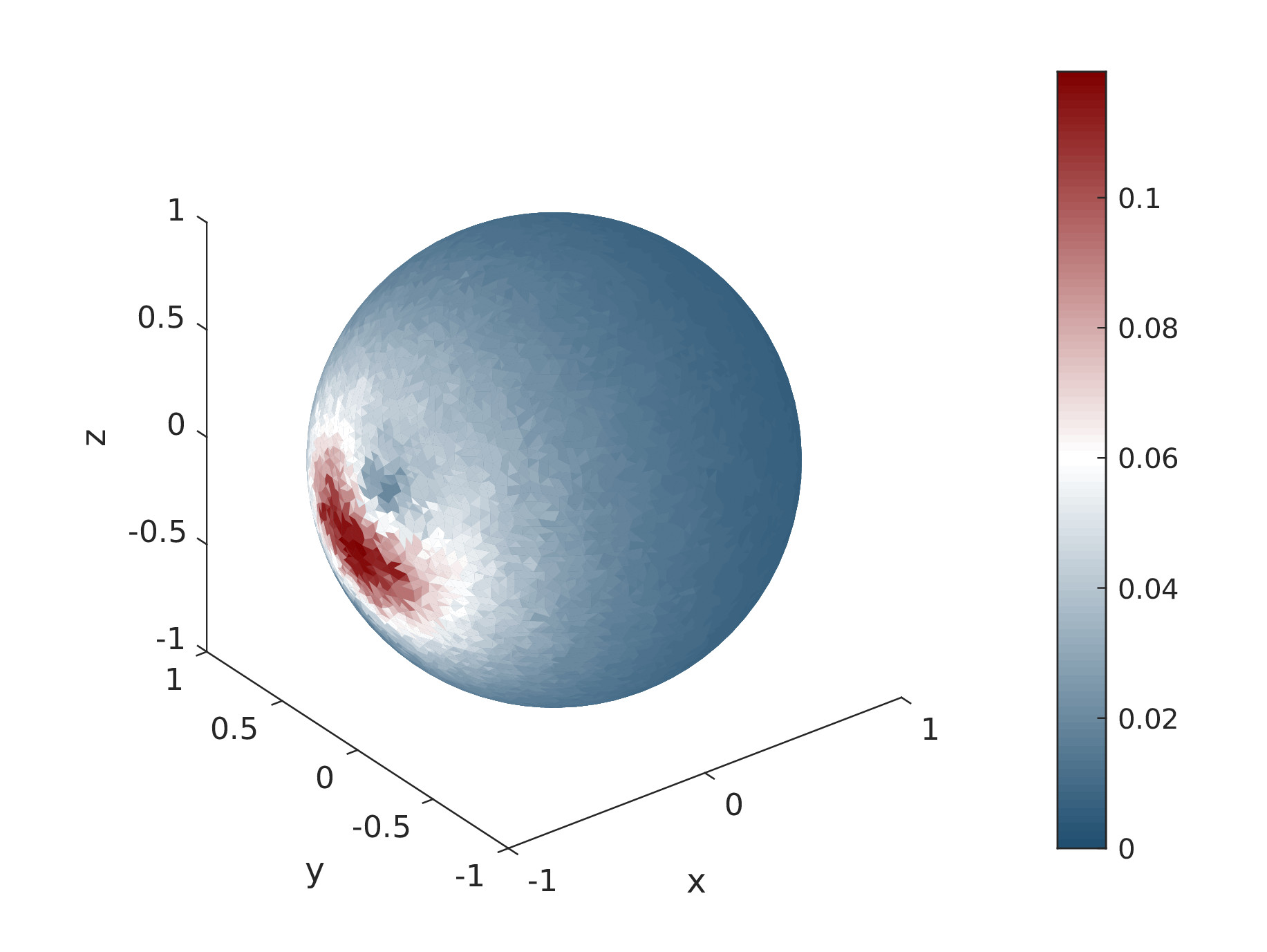}
		\hfill
		\includegraphics[width=0.45\textwidth]{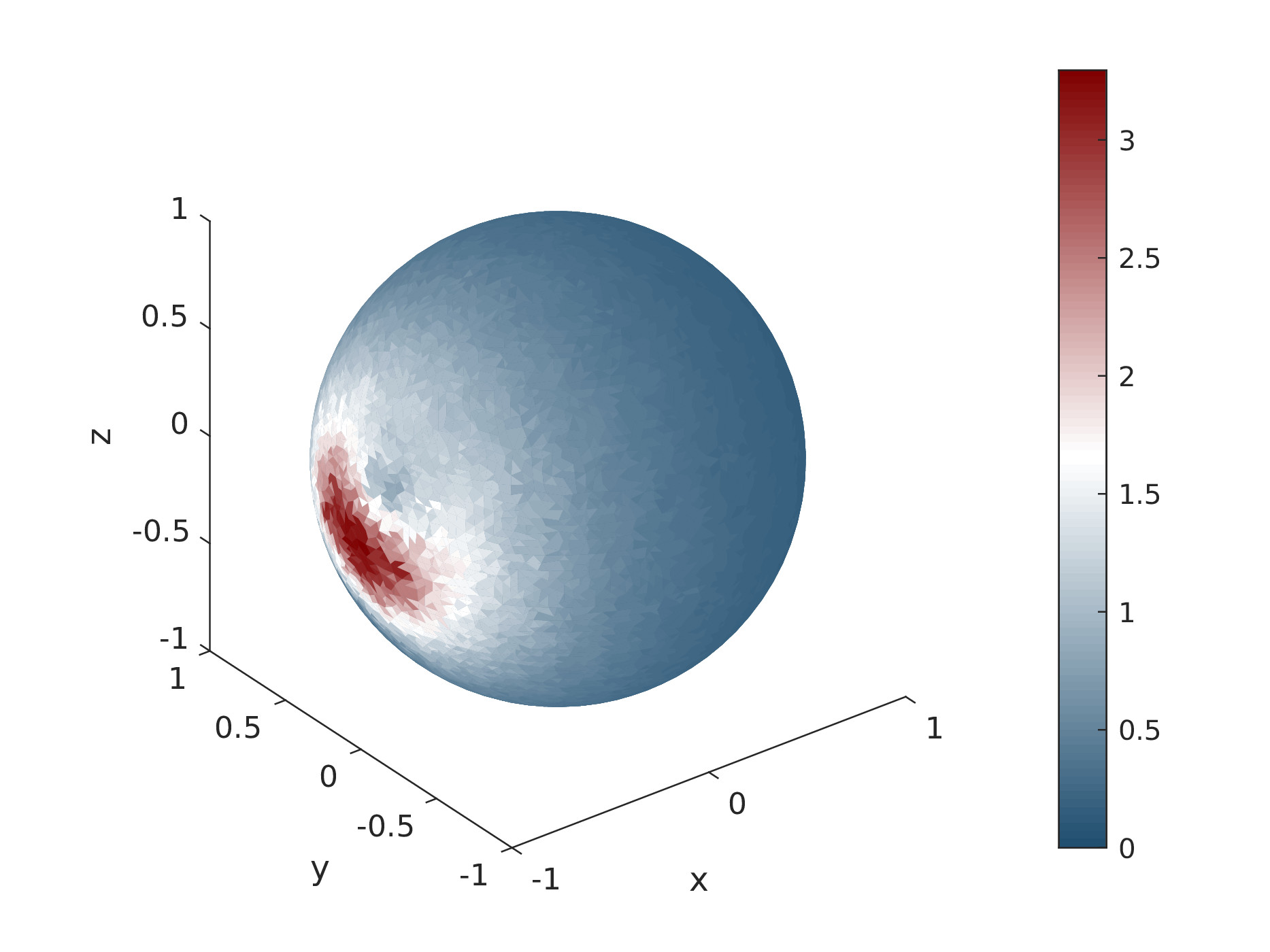}
		\caption{Sensitivity of the electric field on the surface with respect to volume of the inhomogeneity (left: $\alpha = 0.1$, right: $\alpha = 0.3$).}
		\label{fig:sensitivity_deep}
	\end{figure}

	In \autoref{fig:sensitivity_xy}, we present the sensitivity corresponding to two spherical inhomogeneities: one centered at $\bfx_0 = (-0.85,0,0)$ of radius $\alpha=0.1$ and the other centered at $(0,-0.7,0)$ of radius $\alpha=0.2$. We retrieve two surfacic perturbations corresponding to each inhomogeneity in agreement with Proposition~\ref{prop:disjoint}.

	\begin{figure}[p]
		\centering
		\includegraphics[width=0.45\textwidth]{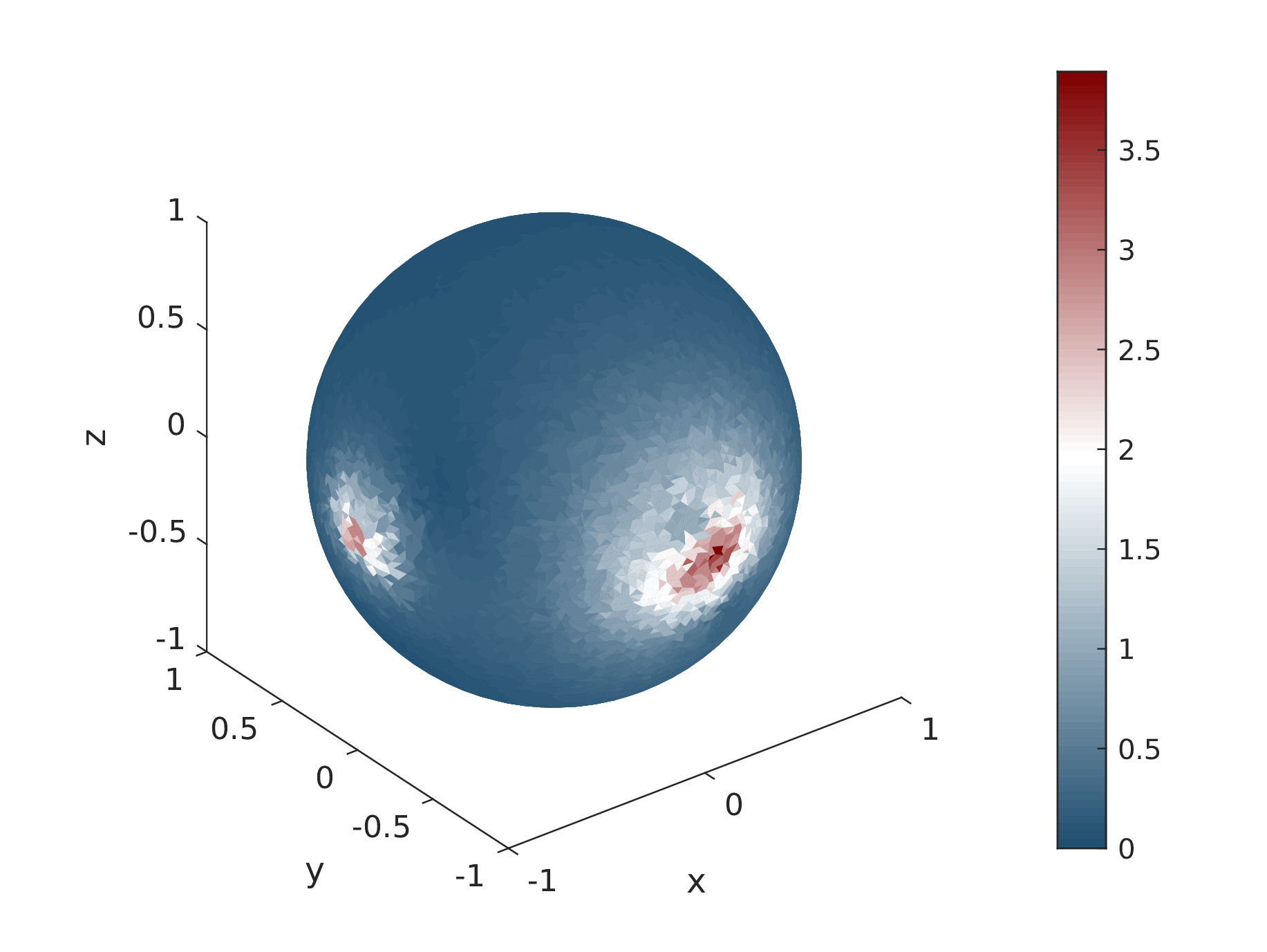}
		\caption{Sensitivity of the electric field on the surface corresponding to two disjoint inhomogeneities.}
		\label{fig:sensitivity_xy}
	\end{figure}

	The numerical results of Figures~\ref{fig:sensitivity_position} to \ref{fig:sensitivity_xy} emphasize that sensitivity analysis provides information about those surface areas on which the electrical field is affected by small variations in the electric parameters of the medium. More precisely, the values of the sensitivity $\bfE^1$ give insights about the inhomogeneities's locations and sizes. We will see in \autoref{sec:inverse} that it is a useful tool for solving the inverse problem of reconstructing the support of a perturbation in the permittivity and/or conductivity from boundary data. Indeed, the solution of the sensitivity equation \eqref{eq:sensitivity} is linked to the perturbed electric field in the following way. Let $\tau \coloneqq (\eps,\sigma) \in \mcP_\textrm{adm}$ and $\varrho = (\varrho_{\eps}, \varrho_{\sigma}) \in \mcP$, both fixed. As suggested in \cite{BorggaardEtiennePelletierEtAl02}, a first order Taylor expansion of the solution $\bfE_p \coloneqq \bfE(\cdot, \tau + h\varrho)$ of the perturbed problem with parameters $\tau+h\varrho$ for small-amplitudes of order $h$, $0 < h \ll 1$, yields
	\begin{equation}
		\label{eq:taylor}
		\bfE(\cdot, \tau + h\varrho) - \bfE(\cdot, \tau) \approx hD_\varrho\bfE(\cdot, \tau).
	\end{equation}
	In other words, for small values of $h$, the boundary data (measurements) $(\bfE_p - \bfE) \times \bfn$ have the same behavior as the Gâteaux derivative of the electric field $\bfE$ in the direction $\varrho$.

	\section{Sensitivity analysis in the case of a constant background}
	\label{sensitivity_sec}

	In this section, we identify the tangential trace of the Gâteaux derivative $\bfE^1$ as the solution of a boundary integral equation. Estimates of the right hand side of this equation exhibit some relations between the sensitivity $\bfE^1$ and geometric characteristics of the perturbation.

	In the sequel, we assume that the material parameters $\eps$ and $\sigma$ of the unperturbed background medium are positive constants. In order to simplify the notations, we introduce the complex-valued wavenumber $\xi$ which is defined by
	\begin{equation}
		\label{eq:k_const}
		\xi^2 \coloneqq k^2 \kappa = \omega^2\mu_0\left(\varepsilon + i\frac{\sigma}{\omega}\right)
	\end{equation}
	where $\varepsilon > 0$ and $\sigma > 0$.

	Let $\bfE$ be the solution of \eqref{eq:maxwell_var} associated with $\xi$. Let $B = B_\alpha(\bfx_0)$ be the sphere of radius $\alpha > 0$ and center $\bfx_0 \in \Omega$. We assume that the distance between $B$ and the boundary $\Gamma$ is at least equal to a given value $\beta > 0$ and we choose a neighborhood $\mcV_\beta(\Gamma)$ of $\Gamma$ such that $\mcV_\beta(\Gamma) \cap B = \emptyset$.

	The perturbation occurs in the domain $B$ and will be described by a function
	\[f\colon \bfx \mapsto \omega^2\mu_0\left(\varrho_{\eps}(\bfx) + i\frac{\varrho_{\sigma}(\bfx)}{\omega}\right)\]
	where $(\varrho_{\eps},\varrho_{\sigma}) \in \mcP$. We assume that $f$ is regular, $f \in W^{2,\infty}(\Omega)$, and that $\supp(f)\subset B$.

	According to \autoref{sec:sensitivity} (see \eqref{eq:sens-edp}), the Gâteaux derivative of $\bfE$ with respect to the (constant) parameters $\tau = (\varepsilon,\sigma)$ is solution of the boundary value problem,
	\begin{equation}
		\label{eq:sens_k_const}
		\left\{
		\begin{array}{rcl@{\hspace{4\tabcolsep}}l}
			\curl\curl\bfE^1 - \xi^2\bfE^1 &=& f\bfE, & \stext[r]{in} \Omega, \\
			\curl\bfE^1 \times \bfn &=& 0, & \stext[r]{on} \Gamma.
		\end{array}
		\right.
	\end{equation}

	\subsection{An integral equation}
	\label{sensitivity_lE}

	In order to state the integral equation for the tangential trace $\bfa = \bfE^1 \times \bfn$ on $\Gamma$, we introduce in the sequel suitable integral operators. Let $\Phi$ denote the fundamental solution in $\R^3$ of the Helmholtz equation with complex wavenumber $\xi$,
	\[-\left(\Delta\Phi + \xi^2\Phi\right) = \delta_0,\]
	satisfying the outgoing Sommerfeld condition as $|\bfx| \to \infty$. Function $\Phi(\bfx)$ is given by
	\[\Phi\colon \bfx \mapsto \frac{1}{4\pi}\frac{e^{i\xi|\bfx|}}{|\bfx|}, \ \bfx \neq 0.\]

	Define the space of continuous tangential fields on $\Gamma$,
	\[\mcT(\Gamma) = \set{\bfa \in \mcC^0(\Gamma)^3}{\bfa \cdot \bfn = 0}.\]
	For $\bfz \in \R^3 \setminus \Gamma$, the vector potential $\bfA(\bfz)$ with density $\bfa \in \mcT(\Gamma)$ is defined by
	\begin{equation}
		\label{eq:vecA}
		\bfA(\bfz) = \int_\Gamma \Phi(\bfx - \bfz)\bfa(\bfx) \dint\bfx.
	\end{equation}
	For the bounded domain $\Omega \subset \R^3$, we denote by $\bfA^-$ the restriction of $\bfA$ to $\Omega$. Similarly, $\bfA^+$ is the restriction of $\bfA$ to the exterior of $\Omega$, $\bfA^+ = \bfA_{\R^3 \setminus \overline{\Omega}}$. The following theorem from \cite{CK98} describes the behavior of $\bfA$ on the boundary $\Gamma$.

	\begin{theorem}
		\label{thm:vecA}
		Assume that $\Omega$ is a domain of class $\mcC^2$ and let $\bfa \in \mcT(\Gamma)$. Then, $\bfA$ is continuous across $\Gamma$, i.e.
		\begin{equation}
			\label{eq:A}
			\bfA(\bfz) = \int_\Gamma \Phi(\bfx - \bfz)\bfa(\bfx) \dint{s}(\bfx), \ \forall \bfz \in \R^3.
		\end{equation}
		Furthermore, the Neumann trace satisfies the jump condition
		\begin{equation}
			\label{eq:curlA}
			\Forall{\bfz \in \Gamma} \curl\bfA^{\pm}(\bfz) \times \bfn = \int_\Gamma \curl_z\left(\Phi(\bfx - \bfz)\bfa(\bfx)\right) \times \bfn_\bfz \dint{s}(\bfx) \mp \frac{1}{2}\bfa(\bfz),
		\end{equation}
		and the following relation holds true uniformly for all $\bfz \in \Gamma$:
		\begin{equation}
			\label{eq:curlcurlA}
			\lim_{h \to 0^+} \left(\curl\curl\bfA^+(\bfz + h\bfn_z) - \curl\curl\bfA^-(\bfz - h\bfn_z)\right) \times \bfn_\bfz = 0.
		\end{equation}
	\end{theorem}

	We next introduce the magnetic dipole operator $\mcM$ which is defined for $\bfa \in \mcT(\Gamma)$ by
	\begin{equation}
		\label{eq:N1}
		\mcM\bfa(\bfz) = 2\int_\Gamma \curl_\bfz \left(\Phi(\bfx - \bfz) \bfa(\bfx)\right) \times \bfn_\bfz \dint{s}(\bfx), \; \forall \bfz \in \Gamma.
	\end{equation}

	Finally, we introduce the fundamental solution of the Maxwell equations that can be derived from $\Phi$ in the following way,
	\begin{equation}
		\label{eq:G}
		G\colon \bfx \mapsto - \Phi(\bfx)\I + \frac{1}{\xi^2}D^2\Phi(\bfx).
	\end{equation}
	Here, $\I \in \mcM_3(\R)$ is the identity matrix, and $D^2\Phi(\bfx)$ denotes the Hessian of $\Phi$ at the point $\bfx$. $G$ can be shown to solve the following equation in $\R^3$,
	\begin{equation}
		\label{eq:Maxwell_delta}
		\curl\curl G - \xi^2G = -\delta_0 \I,
	\end{equation}
	where the curl of the matrix valued function $G$ has to be understood column wise \cite{Ned00}. Since $\Phi$ satisfies the outgoing radiation condition, $G$ satisfies the following Silver-Müller condition \cite{CK98}
	\[
		\lim_{|\bfx| \to \infty} |\bfx|\left(\curl G \times \dfrac{\bfx}{|\bfx|} - i\xi G\right) = 0.
	\]
	Then, we are able to state the following theorem.

	\begin{theorem}
		\label{thm:EI}
		Let $\bfE^1$ be the solution of \eqref{eq:sens_k_const} for $\xi$ given by \eqref{eq:k_const} with $\varepsilon > 0$ and $\sigma > 0$. For $\bfz \in \Gamma$, define $\bfT(\bfz)$ by
		\begin{equation}
			\label{eq:def_T}
			\bfT(\bfz) = -2\left(\int_\Omega G(\bfx -\bfz)f(\bfx)\bfE(\bfx) \dint\bfx\right) \times \bfn.
		\end{equation}
		Under the regularity assumptions of \autoref{thm:sens-reg2}, the tangential trace $\bfa = \bfE^1 \times \bfn$ is solution of the following integral equation on $\Gamma$,
		\begin{equation}
			\label{eq:EI}
			(\mcI - \mcM)\bfa = \bfT,
		\end{equation}
		where $\mcI$ denotes the identity operator.
	\end{theorem}

	The proof of Theorem \ref{thm:EI} is adapted from \cite{AVV01} where an asymptotic expansion of the perturbed field is obtained in terms of the (small) radius of the perturbation. Notice that in the present study of sensitivity, the analysis is simplified and no asymptotic parameter occurs. In other words, we do not assume that the perturbation is small in size, but only in amplitude in order to connect the derivative to the perturbed field \eqref{eq:taylor}.

	\begin{proof}
		Let $\bfz \in \mcV_\beta(\Gamma)$. According to \autoref{thm:sens-reg2}, the solution of the sensitivity equation \eqref{eq:sensitivity} belongs to $H^2(\Omega)^3 \hookrightarrow \mcC^0(\overline{\Omega})^3$, and thus the duality product $\duality{\delta_0\I}{\bfE^1}{} = \bfE^1(\bfz)$ is well defined. From \eqref{eq:Maxwell_delta}, we get
		\begin{equation}
			\label{eq:EI-1}
			-\bfE^1(\bfz) = \int_\Omega \curl_\bfx\curl_\bfx G(\bfx -\bfz)\bfE^1(\bfx) \dint\bfx - \xi^2\int_\Omega G(\bfx - \bfz)\bfE^1(\bfx) \dint\bfx,
		\end{equation}
		where the integrals have been to understood as duality products in $H^s(\Omega)$ for appropriate values of $s$.

		The following partial integration formula holds true for matrix valued functions $A\colon \Omega \to \C^{3 \times 3}$ and $B\colon \Omega \to \C^{3 \times p}$, $p \in \N^*$,
		\begin{equation}
			\label{eq:EI_green}
			\int_\Omega \transpose{A} \curl B - \int_\Omega \transpose{(\curl A)}B = \int_\Gamma \transpose{(A \times \bfn)}B,
		\end{equation}
		where the vector product $A \times \bfn$ is taken column wise. Applying \eqref{eq:EI_green} twice to the first term on the right hand side of \eqref{eq:EI-1} yields
		\[
			\bfE^1(\bfz) = \int_\Gamma \transpose{(\curl_\bfx G(\bfx -\bfz) \times \bfn)}\bfE^1(\bfx) \dint{s}(\bfx) - \int_\Omega G(\bfx -\bfz)f(\bfx)\bfE(\bfx) \dint\bfx
		\]
		taking into account the symmetry of either $G$ and $\curl\curl G$ as well as the strong formulation of the sensitivity equation \eqref{eq:sens_k_const}. The remaining boundary integral on the right hand side can be written in terms of the tangential trace of $\bfE^1$ taking into account that $\curl(D^2\Phi) = 0$ in the definition of $G$. Finally, the following identity holds true for any $\bfz$ in the neighborhood $\mcV_\beta(\Gamma)$ of the boundary,
		\begin{equation}
			\label{eq:EI-2}
			\bfE^1(\bfz) - \int_\Gamma \curl_\bfz(\Phi(\bfx -\bfz)(\bfE^1(\bfx) \times \bfn)) \dint{s}(\bfx) = -\int_\Omega G(\bfx - \bfz)f(\bfx)\bfE(\bfx) \dint\bfx.
		\end{equation}
		Notice that the volume integral on $\Omega$ is well defined since $\supp(f)\subset B$ and $\bfz \in \mcV_\beta(\Gamma)$. Hence, $\bfx \neq \bfz$ for any $\bfx \in B$ and $G$ is regular on the integration domain.

		Now, for $\bfz \in \mcV_\beta(\Gamma)$, let $\bfz_\Gamma$ denote the projection of $\bfz$ onto $\Gamma$. Notice that $\bfz_\Gamma$ is well defined since $\Gamma$ is regular and $\mcV_\beta(\Gamma)$ can be assumed to be sufficiently small. Taking the vector product in the identity \eqref{eq:EI-2} with the vector $\bfn_\bfz = \bfn_{\bfz_\Gamma}$ and passing to the limit as $\bfz \to \bfz_\Gamma \in \Gamma$ yields the integral equation \eqref{eq:EI} on $\Gamma$ according to \eqref{eq:curlA}.
	\end{proof}

	In order to obtain an estimate of $\bfE^1\times\bfn$, we analyze the operator involved in the integral equation~\eqref{eq:EI}. We introduce the following normed spaces of tangential fields with surface divergence
	\begin{align*}
		\mcT_d(\Gamma) &= \set*{\bfa \in \mcT(\Gamma)}{\div_\Gamma\bfa \in \mcC^0(\Gamma)} \\
		\shortintertext{and}
		\mcT^{0,\alpha}_d(\Gamma) &= \set*{\bfa\in \mcT(\Gamma)}{\bfa \in \mcC^{0,\alpha}(\Gamma); \div_\Gamma\bfa \in \mcC^{0,\alpha}(\Gamma)}
	\end{align*}
	equipped with the respective graph norms.

	\begin{theorem}
		\label{thm:EI-bij}
		Under the assumptions of \autoref{thm:EI}, the operator $\mcI - \mcM$ is bijective on the space $\mcT_d(\Gamma)$ and has a bounded inverse.
	\end{theorem}

	\begin{proof}
		According to \cite[Theorems 6.15 and 6.16]{CK98}, we can state that the operator $\mcM\colon \mcT_d(\Gamma) \to \mcT_d^{0,\alpha}(\Gamma)$ is continuous, whereas $\mcT_d^{0,\alpha}(\Gamma)$ is compactly embedded in $\mcT_d(\Gamma)$. Hence, $\mcM$ is a compact operator on $\mcT_d(\Gamma)$.

		We prove that $\mcI - \mcM$ is injective. To this end, let $\bfa \in \mcT_d(\Gamma)$ and define the vector field $\bfE_\bfa$ for $\bfz \in \R^3 \setminus \Gamma$ by
		\[\bfE_\bfa(\bfz) = \curl\int_\Gamma\Phi(\bfx - \bfz)\bfa(\bfx) \dint{s}(\bfx) = \curl\bfA(\bfz).\]
		Here $\bfA$ is the vector potential with density $\bfa$ introduced in \eqref{eq:A}. We have
		\[\curl\curl\bfE_\bfa^\pm - \xi^2\bfE_\bfa^\pm = 0 \stext{in} \Omega^\pm\]
		as well as
		\begin{equation}
			\label{eq:EI-3}
			\bfE_\bfa^\pm \times \bfn = \curl\bfA^\pm \times \bfn = \frac{1}{2}\left(\mcM \mp \mcI \right)\bfa,\ \forall \bfz \in \Gamma.
		\end{equation}
		Now, let $\bfa \in \mcT_d(\Gamma)$ such that $\left(\mcI - \mcM\right)\bfa = 0$. On the exterior domain $\Omega^+ = \R^3 \setminus \overline{\Omega}$, $\bfE_\bfa^+$ is solution of the exterior Maxwell problem with homogeneous Dirichlet boundary condition. In addition, $\bfE^+$ satisfies the outgoing radiation condition at infinity as does the fundamental solution $\Phi$. Consequently, $\bfE_\bfa^+ \equiv 0$ on $\Omega^+$ due to the uniqueness of the solution to the exterior Maxwell problem.

		Next, applying identity \eqref{eq:curlcurlA} to $\curl\bfE_\bfa = \curl\curl\bfA$, we get
		\[\lim_{h \to 0^+} \left(\curl\bfE_\bfa^+(\bfz + h\bfn_\bfz) - \curl\bfE_\bfa^-(\bfz - h\bfn_\bfz)\right) \times \bfn_\bfz = 0.\]
		But $\bfE_\bfa^+$ vanishes on $\Omega^+$ and therefore,
		\[\curl\bfE_\bfa^- \times \bfn = 0 \stext{on} \Gamma.\]
		The field $\bfE_\bfa^-$ is thus solution of the interior Maxwell problem with homogeneous Neumann boundary condition.

		According to the properties of the constant parameters $\varepsilon > 0$ and $\sigma > 0$ in the definition of the wave number $\xi$, the only solution to the interior Maxwell problem is $\bfE_\bfa^- \equiv 0$ (see \autoref{thm:sol_maxwell}).

		From the homogeneous integral equation $(\mcI - \mcM)\bfa = 0$, we deduce $\mcM\bfa = \bfa$. Together with the identity \eqref{eq:EI-3}, this yields
		\[\bfE_\bfa^- \times \bfn = \frac{1}{2}(\mcM + \mcI)\bfa = \bfa\]
		and thus $\bfa = 0$ on $\Gamma$. This proves that the operator $\mcI-\mcM$ is injective. Since $\mcM$ has been shown to be compact on $\mcT_d(\Gamma)$, we deduce from the Fredholm alternative that $\mcI - \mcM$ is bijective. Finally, $\mcI - \mcM$ has a bounded inverse according to the inverse (or open) mapping theorem.
	\end{proof}

	Theorems~\ref{thm:EI} and~\ref{thm:EI-bij} imply that the tangential trace $\bfa = \bfE^1 \times \bfn$, solution to the integral equation
	\[(\mcI - \mcM)\bfa = \bfT,\]
	can be estimated by
	\begin{equation}
		\label{estimation_sensitivity}
		\norm{\bfa}{\mcT_d(\Gamma)} \leq C_\mcM\norm{\bfT}{\mcT_d(\Gamma)}.
	\end{equation}

	\begin{remark}
		The norm of the inverse operator in \eqref{estimation_sensitivity} actually depends on the wavenumber $\xi$. This may be seen from a thorough analysis of the eigenvalues of the operator $\mcM$ when $\Omega$ is a sphere. In this case, an exact analytical expression of the eigenvalues can be obtained in function of Ricatti-Bessel and Ricatti-Hankel functions~\cite{HsiaoKleimann}. This study allows us to numerically observe the behavior of the spectrum of the operator $(\mcI - \mcM)$ and its inverse with respect to the wavenumber $\xi$. The operator $(\mcI - \mcM)$ is called the MFIE (Magnetic Field Integral Equation) operator~\cite{Ned00}.
	\end{remark}

	\subsection{Estimates for \texorpdfstring{$\bfT$}{T}}
	\label{propertyT}

	In this section, we will prove some estimates of the functional $\bfT$ on the boundary that are at the origin of the localization algorithm described in Section~\ref{sec:algo}. The proof is mainly based on the following estimates of the fundamental solution $G$.

	\begin{lemma}
		\label{lemma:estimate_G}
		Let the wavenumber $\xi \in \C$ be defined as in \eqref{eq:k_const} and let $G$ be the associated fundamental solution of the Maxwell equations as in \eqref{eq:G}. There is a polynomial $p$ of degree 3 with positive coefficients depending on $\abs{\xi}$ satisfying $p(0) = 0$ such that
		\begin{equation}
			\label{eq:estimate_G}
			\abs{G_{j\ell}(\bfx)} \leq p\left(\frac{1}{\abs{\bfx}}\right)\ \forall 1 \leq j, \ell \leq 3, \forall \bfx \neq 0.
		\end{equation}
		Similarly, there is a polynomial $q \in \mathbb{P}_4(\R)$ of degree 4 with positive coefficients depending on $\abs{\xi}$ satisfying $q(0) = 0$, such that
		\begin{equation}
			\label{eq:estimate_dmG}
			\abs{\partial_m G_{j\ell}(\bfx)} \leq q\left(\frac{1}{|\bfx|}\right)\ \forall 1 \leq j, \ell, m \leq 3.
		\end{equation}
	\end{lemma}

	\begin{proof}
		Let $1 \leq j,\ell \leq 3$ and $\bfx \neq 0$. We have
		\[G_{j\ell}(\bfx) = -\frac{1}{4\pi}\frac{e^{i\xi|\bfx|}}{|\bfx|}\delta_{j\ell} + \frac{1}{\xi^2}\partial_j\partial_\ell\Phi(\bfx).\]
		A straightforward computation of the derivatives yields
		\[G_{j\ell}(\bfx) = \frac{e^{i\xi|\bfx|}}{4\pi}\left(\left(-\frac{1}{|\bfx|}+\frac{i}{\xi|\bfx|^2} - \frac{1}{\xi^2|\bfx|^3}\right)\delta_{j\ell} - \left( \frac{1}{|\bfx|}+\frac{3i}{\xi|\bfx|^2} - \frac{3}{\xi^2|\bfx|^3}\right)\frac{x_jx_\ell}{|\bfx|^2}\right)\]
		which can be estimated by
		\[\abs{G_{j\ell}(\bfx)} \leq \frac{1}{4\pi}\left(\frac{2}{|\bfx|} + \frac{4}{|\xi| |\bfx|^2} + \frac{4}{|\xi|^2|\bfx|^3}\right).\]
		This yields \eqref{eq:estimate_G} where the coefficients of the polynomial $p$ depend on the wavenumber $\xi$. Estimate \eqref{eq:estimate_dmG} follows in the same way.
	\end{proof}

	\begin{theorem}
		\label{theorem:ineq_T}
		Let $B = B_\alpha(\bfx_0)$ be the sphere of radius $\alpha > 0$ and center $\bfx_0 \in \Omega$ and assume that $\dist(B,\Gamma) \geq \beta > 0$. Then there is a constant $C = C(\beta, f, \bfE) > 0$ such that for any $\bfz \in \Gamma$,
		\begin{align}
			\label{eq:ineq_T}
			\abs{\bfT(\bfz)} &\leq Cp\left(\frac{1}{\abs{\bfx_0 - \bfz}}\right)\vol(B) \\
			\shortintertext{and}
			\label{eq:ineq_divT}
			\abs{\div_{\Gamma}\bfT(\bfz)} &\leq Cq\left(\frac{1}{\abs{\bfx_0-\bfz}}\right)\vol(B),
		\end{align}
		where $p$ and $q$ are the polynomials from \autoref{lemma:estimate_G}.
	\end{theorem}

	\begin{proof}
		We recall that
		\[\bfT(\bfz) = -2 \left(\int_\Omega G(\bfx -\bfz)f(\bfx)\bfE(\bfx) \dint\bfx\right) \times \bfn.\]
		Since the background parameters $\eps$ and $\sigma$ are constant, the solution $\bfE$ of problem~\eqref{eq:maxwell_var} belongs to $H^2(\Omega) \hookrightarrow \mcC^0(\overline{\Omega})$ and $\norm{\bfE}{\infty,B}$ is well defined. Taking into account that $\supp(f)\subset B$, we get
		\[|\bfT(\bfz)| \leq 2 \norm{f}{\infty,B} \max_{1 \leq j,\ell \leq 3} \norm{G(\cdot-\bfz)}{\infty,B} \norm{\bfE}{\infty,B}\vol(B).\]
		Now, let $\bfx \in B$ and $\bfz \in \Gamma$. According to the assumptions on $B$, we have $|\bfx - \bfz| \geq \beta > 0$. Moreover, since $|\bfx_0 - \bfx| \leq \alpha$, we get
		\[|\bfx_0 - \bfz| \leq \alpha + |\bfx - \bfz| \leq |\bfx - \bfz| \left(\frac{\alpha}{|\bfx - \bfz|} + 1\right) \leq |\bfx - \bfz| \left(\frac{\alpha}{\beta} + 1\right).\]
		The constant $\frac{\alpha}{\beta} + 1$ can be majored for all possible values of $\alpha$ by $\frac{\diam(\Omega)}{2\beta}+1$. Noticing that the polynomial $p$ in \autoref{lemma:estimate_G} has positive coefficients allows to write
		\[p\left(\frac{1}{|\bfx-\bfz|}\right) \leq Cp\left(\frac{1}{|\bfx_0 -\bfz|}\right)\]
		and completes the proof of the estimate~\eqref{eq:ineq_T}.

		In order to obtain the estimate for the surface divergence of $\bfT$, we notice that
		\[\div_\Gamma \bfT(\bfz) = \bfn \cdot (\curl\bfT)(\bfz),\]
		for any $\bfz \in \Gamma$. But the computation of $\curl\bfT$ involves the first order derivatives of the fundamental solution $G(\cdot - \bfz)$ which are estimated with the help of the polynomial $q$ (see \eqref{eq:estimate_dmG}). This yields \eqref{eq:ineq_divT} noticing again that $q$ has positive coefficients.
	\end{proof}

	\subsection{Estimates for \texorpdfstring{$\bfE^1 \times \bfn$}{E1 x n}}
	\label{ss:propertyE1}

	We deduce from the previous section the following estimates that yield relations between the tangential trace of the sensitivity and caracteristics of the perturbation located in the ball $B$. As before, let $B = B_\alpha(\bfx_0)$ be the ball of radius $\alpha$ and center $\bfx_0 \in \Omega$. We assume that $\dist(B,\Gamma) \geq \beta$ for a fixed constant $\beta > 0$. We denote by $\hat{\bfx}$ the projection of $\bfx_0$ on the boundary $\Gamma$.

	\begin{proposition}
		\label{p:volume}
		Let $\bfE^1$ denote the solution of the sensitivity equation~\eqref{eq:sensitivity}. Under the assumptions of \autoref{theorem:ineq_T}, we have
		\begin{equation}
			\label{rel:volume}
			\norm{\bfE^1 \times \bfn}{0,\Gamma} \leq C\left(p\left(\frac{1}{|\bfx_0-\hat{\bfx}|}\right) + q\left(\frac{1}{|\bfx_0-\hat{\bfx}|}\right)\right)\vol(B),
		\end{equation}
		where $p$ and $q$ are the polynomials of \autoref{lemma:estimate_G}, and $C > 0$ is a constant independent from $\alpha$ and $d$.
	\end{proposition}

	\begin{proof}
		First notice that $\norm{\bfE^1 \times \bfn}{0,\Gamma} \leq \area(\Gamma)^{1/2} \norm{\bfE^1 \times \bfn}{\infty,\Gamma}$ since $\bfE^1$ is continuous on $\overline{\Omega}$ according to the regularity assumptions. We next recall that
		\[\norm{\bfE^1 \times \bfn}{\infty,\Omega} \leq C_\mcM \norm{\bfT}{\mcT_d(\Gamma)}\]
		where $C_\mcM$ denotes the norm of $(\mcI - \mcM)^{-1}$. According to \autoref{theorem:ineq_T}, $|\bfT(\bfz)|$ (resp. $|\div_\Gamma\bfT(\bfz)|$) can be estimated for any $\bfz \in \Gamma$ by the polynomial $p$ (resp. $q$) which has positive coefficients and satisfies $p(0) = 0$ (resp. $q(0) = 0$). Notice further that $\frac{1}{|\bfx_0-\bfz|}$ takes its maximum value for $\bfz = \hat{\bfx}$, and so do $p\left(\frac{1}{|\bfx_0-\bfz|}\right)$ and $q\left(\frac{1}{|\bfx_0-\bfz|}\right)$. Then, \eqref{rel:volume} follows from \eqref{eq:ineq_T} and \eqref{eq:ineq_divT}.
	\end{proof}

	The next estimate states that $\bfE^1 \times \bfn$ behaves similar to $\bfT$ on the boundary $\Gamma$:
	\begin{proposition}
		\label{p:rel_proj}
		Under the assumptions of \autoref{theorem:ineq_T}, there are constants $\widetilde{C} > 0$ and $c > 0$ independent from $\bfz$ and the ball $B = B_\alpha(\bfx_0)$ such that for any $\bfz \in \Gamma$,
		\begin{equation}
			\label{eq:rel_proj}
			|(\bfE^1\times\bfn)(\bfz)| \leq \frac{\widetilde{C}}{|\bfx_0-\bfz|} + c.
		\end{equation}
	\end{proposition}

	\begin{proof}
		Let $\bfa = \bfE^1 \times \bfn = (\mcI - \mcM)^{-1}\bfT$. The following identity can be easily verified,
		\[(\mcI - \mcM)^{-1}\bfT = \bfT + (\mcI - \mcM)^{-1}\mcM\bfT.\]
		Now, define the constant $K_\mcM = \norm{(\mcI - \mcM)^{-1}\mcM}{\mcT_d(\Gamma)}$. One gets
		\[\frac{1}{K_\mcM}\abs{\bfa(\bfz)} \leq \frac{1}{K_\mcM}\abs{\bfT(\bfz)} + \frac{1}{K_\mcM} \norm{(\mcI - \mcM)^{-1}\mcM}{\mcT_d(\Gamma)} \norm{\bfT}{\mcT_d(\Gamma)}\]
		or, equivalently,
		\[\abs{\bfa(\bfz)} \leq \abs{\bfT(\bfz)} + K_\mcM \norm{\bfT}{\mcT_d(\Gamma)}.\]
		An estimate for $|\bfT(\bfz)|$ has been obtained in \autoref{theorem:ineq_T}. Here, we only keep the dominating terms. Since $|\bfx_0 - \bfz| \geq \beta$, we have
		\[\frac{1}{|\bfx_0 - \bfz|^n} \leq \frac{1}{\beta^{n-1}} \frac{1}{|\bfx_0 - \bfz|}\]
		for any $n \geq 1$ which yields the first term on the right hand side of \eqref{eq:rel_proj}.

		In order to get an estimate of $\norm{\bfT}{\mcT_d(\Gamma)}$, we state as in the proof of \autoref{p:volume} that $p\left(\frac{1}{|\bfx_0 - \bfz|}\right)$ and $q\left(\frac{1}{|\bfx_0 - \bfz|}\right)$ reach their maximum values at $\bfz = \hat{\bfx}$. Therefore, we have
		\[\norm{\bfT}{\mcT_d(\Gamma)} \leq \frac{C_1}{|\bfx_0 - \hat{\bfx}|}\]
		and the right hand side of the above inequality can be majored by the constant $c = \frac{C_1}{\beta} > 0$ independently from the ball $B$.
	\end{proof}

	\section{The sensitivity analysis for solving an inverse problem}
	\label{sec:inverse}

	The inverse medium problem, that we are interested in, is to localize inhomogeneities in the electrical parameters of the medium from total or partial boundary data on $\Gamma$ for a given (boundary) source term, at a fixed frequency $\omega$. The setting is similar to the ones in \cite{Ammari03, DarbasLohrengel, deBuhanDarbas}. Our inverse method is based on the informations obtained by the sensitivity analysis.

	\subsection{An inverse medium problem}

	We assume that $\Omega$ is filled with a medium of electrical permittivity and conductivity
	\begin{equation}
		\label{def_perturbations}
		\eps_p = \eps + a_\eps \varrho_\eps \stext{and} \sigma_p = \sigma + a_\sigma \varrho_\sigma
	\end{equation}
	where $\varrho_\eps$ and $\varrho_\sigma$ are the characteristic functions of a perturbation in the homogeneous background parameters $\eps$ and $\sigma$.

	Let $\bfE_\bfeta$ be a plane wave of direction $\bfeta \in \R^3$,
	\[\bfE_\bfeta(\bfx) = \bfeta^\perp e^{i\bfeta\cdot\bfx}.\]
	Here, $\bfeta^\perp$ is a unit vector orthogonal to $\bfeta$. $\bfE_\bfeta$ is acting as a boundary source term for the Neumann trace. Notice that other source terms could have been considered. In the absence of inhomogeneities, the electric field $\bfE$ is solution to
	\begin{equation}
		\label{eq:inverse_E}
		\left\{
		\begin{array}{rcl@{\hspace{4\tabcolsep}}l}
			\curl\curl\bfE - k^2\dfrac{1}{\eps_0}\left(\eps + i\dfrac{\sigma}{\omega}\right)\bfE &=& 0, & \stext[r]{in} \Omega, \\
			\curl\bfE \times \bfn &=& \curl\bfE_\bfeta \times \bfn, & \stext[r]{on} \Gamma.
		\end{array}
		\right.
	\end{equation}

	Next, consider the electric field $\bfE_p$ in the presence of the inhomogeneities and subject to the same boundary data. $\bfE_p$ is solution to the perturbed problem
	\begin{equation}
		\label{eq:inverse_Eh}
		\left\{
		\begin{array}{rcl@{\hspace{4\tabcolsep}}l}
			\curl\curl\bfE_p - k^2\dfrac{1}{\eps_0}\left(\eps_p + i\dfrac{\sigma_p}{\omega}\right)\bfE_p &=& 0, & \stext[r]{in} \Omega, \\
			\curl\bfE_p \times \bfn &=& \curl\bfE_\bfeta \times \bfn, & \stext[r]{on} \Gamma.
		\end{array}
		\right.
	\end{equation}

	We focus on perturbations of small amplitude and simple geometries (sphere, ellipsoid, \ldots). The inverse problem consists in retrieving their centers and volumes from boundary data $(\bfE_p - \bfE) \times \bfn$. According to Taylor expansion~\eqref{eq:taylor}, the boundary data are related to the sensitivity data $\bfE^1 \times \bfn$ of the electric field with respect to these perturbations. The estimates of \autoref{sensitivity_sec}, completed by a numerical study, allow to find three explicit relations between the data and the characteristics of the inhomogeneities. These relations and their link to the theoretical estimates are presented in \autoref{sec:relations}. They have been validated numerically for a large number of configurations and the results of this verification are presented in \autoref{sec:conjectures}.

	\subsection{Explicit relations between data and inhomogeneities}
	\label{sec:relations}

	We infer from the results of \autoref{ss:propertyE1} that the boundary sensitivity data $\bfE^1 \times \bfn$ behave approximately as follows,
	\begin{equation}
		\label{R1}
		|(\bfE^1 \times \bfn)(\bfz)| \approx \frac{C}{|\bfz - \bfx_0|} + c, \forall \bfz \in \Gamma,
		\tag{R1}
	\end{equation}
	where $C$ and $c$ are (unknown) positive constants. In other words, the modulus of the data on the boundary takes its maximum value at the projection $\hat{\bfx}$ of the perturbation's center $\bfx_0$ and should be small far away from $\hat{\bfx}$. This allows to retrieve the position of the projection $\hat{\bfx}$ from the modulus of the data (see \autoref{fig:localization_schema}).

	\begin{figure}[hbt]
		\centering
		\includegraphics[width=0.3\textwidth]{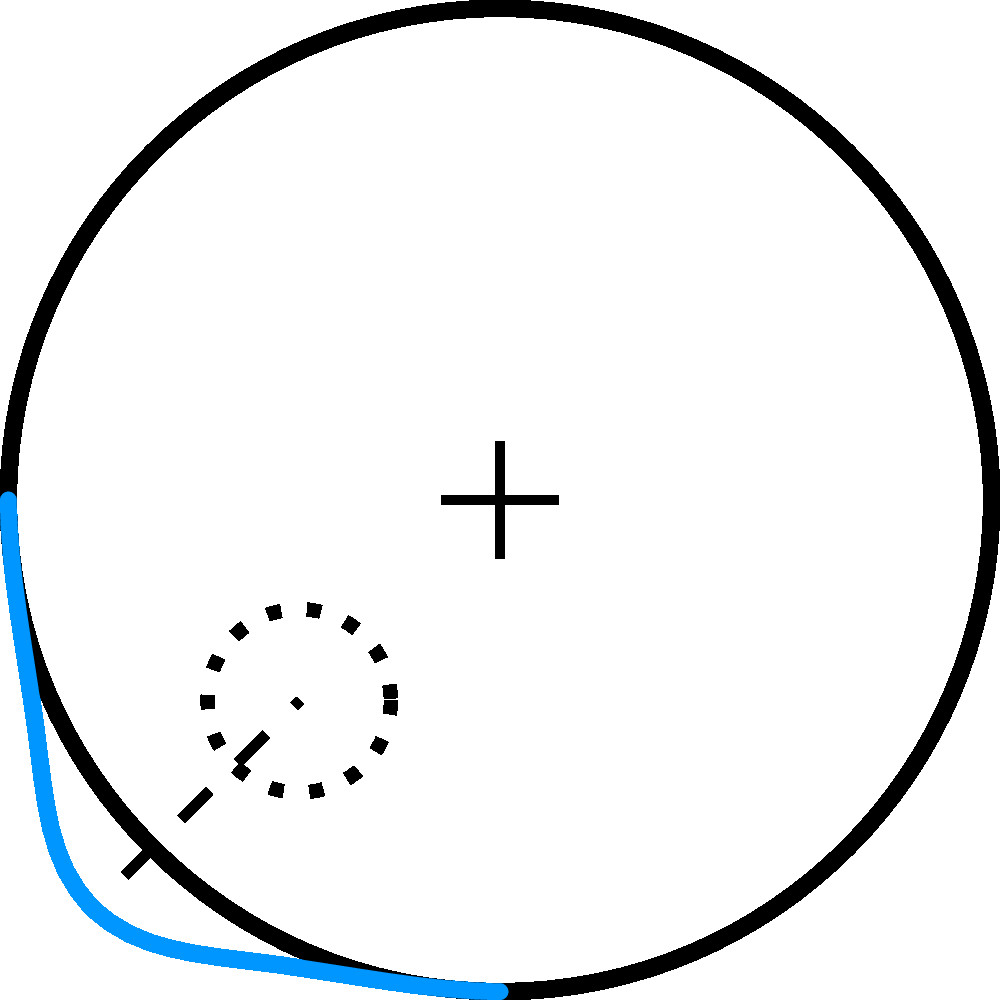}
		\caption{Illustration of the relation~\eqref{R1}. The part of the boundary $\Gamma$ where the largest values of the modulus of the data are reached.}
		\label{fig:localization_schema}
	\end{figure}

	Next, we aim to reconstruct the depth of the perturbation. Together with the projection $\hat{\bfx}$ obtained in the previous step, this yields the center $\bfx_0$. To this end, let $B = B_\alpha(\bfx_0)$ be the ball of radius $\alpha$ and center $\bfx_0$ and denote by $d = |\bfx_0 - \hat{\bfx}|$ the distance of the center of the perturbation to its projection on the boundary. Let $\bfE_d^1$ be the sensitivity in the direction $\varrho_d = (0,1_B)$. We introduce the set
	\begin{equation}
		\label{defGamma}
		\Gamma_{\vartheta}(d) = \set*{\bfz \in \Gamma}{|(\bfE^1_d \times \bfn)(\bfz)| \geq \vartheta\norm{\bfE^1_d \times \bfn}{\infty,\Gamma}},
	\end{equation}
	where $0 < \vartheta < 1$ is a fixed threshold. The following relation has been obtained from numerical simulations,
	\begin{equation}
		\label{R2}
		\frac{\area(\Gamma_{\vartheta}(d))}{\area(\Gamma)} \approx \frac{1}{1 + e^{p_\vartheta(d)}}
		\tag{R2}
	\end{equation}
	where $p_\vartheta$ is a (known) polynomial function of degree 4 that is independent from the radius $\alpha$ of the perturbation (see \autoref{fig:depth_areas}). Since the left hand side of relation~\eqref{R2} can be computed from the boundary data $\bfE^1_d \times \bfn$, relation~\eqref{R2} allows to compute the depth $d$ by inversion of the function $1/(1+e^{p_\vartheta(d)})$.

	\begin{remark}
		Relation~\eqref{R2} could be interpreted in the following probabilistic way. Assume that the boundary point $\bfz \in \Gamma$ is chosen randomly following a uniform distribution. Then, the term
		\begin{equation}
			\label{eq:lhsR2}
			\frac{|(\bfE^1\times\bfn)(\bfz)|}{\norm{\bfE^1 \times \bfn}{\infty,\Gamma}}
		\end{equation}
		may be interpreted as a random variable $X$ that follows a probabilistic law described by a density function $f_X(\bfz;d)$ depending on the depth $d$. Consequently, the left hand side of relation~\eqref{R2} is given by
		\[\frac{\area(\Gamma_{\vartheta}(d))}{\area(\Gamma)} \approx \mathbb{P}(X \geq \vartheta) = 1 - F_X(\vartheta; d)\]
		where $F_X(\vartheta; d)$ is the cumulative distribution function associated with the density $f_X(\bfz; d)$. Now, we infer from \autoref{ss:propertyE1}, that on the one hand, $|(\bfE^1\times\bfn)(\bfz)|$ takes its maximum at the point $\bfz = \hat{\bfx}$, and, on the other, $\norm{\bfE^1 \times \bfn}{\infty,\Gamma}$ behaves roughly speaking as $1/|\bfx_0 - \hat{\bfx}| = 1/d$. If we assume that the random variable $X$ follows a logistic law (which is consistent with the theoretical results),
		\[X \thicksim \frac{e^{-x/d}}{d(1+e^{-x/d})^2}\]
		we get
		\[\frac{\area(\Gamma_{\vartheta}(d))}{\area(\Gamma)} = 1 - \frac{1}{1 + e^{-\vartheta/d}} = \frac{1}{1+e^{\vartheta/d}}.\]
		We may notice that this behavior fits qualitatively with the numerical observations (see \autoref{fig:depth_areas}). For a better concordance with the numerical results, however, the quantity $1 - F_X$ has been fitted with the help of a polynomial function $p_{\vartheta}$ such that
		\[1 - F_X(\vartheta; d) \approx \frac{1}{1 + e^{p_\vartheta(d)}}.\]
		This yields relation~\eqref{R2}. Notice also that both the numerator and the denominator in \eqref{eq:lhsR2} depend linearly on the volume of the perturbation. Consequently, relation~\eqref{R2} should not behave on $\vol(B)$ which is confirmed by the numerical results.
	\end{remark}

	Finally, we aim to obtain the volume of the perturbation. To this end, we recall that according to \eqref{rel:volume}, the $L^2$-norm of the boundary data $\bfE^1 \times \bfn$ is related to $\vol(B)$ by a linear relation with a constant depending on $1/|\bfx_0 - \hat{\bfx}| = 1/d$. This constant can be fitted numerically and leads to the following relation between the data and the volume,
	\begin{equation}
		\label{R3}
		\norm{\bfE^1 \times \bfn}{0,\Gamma} \approx e^{p(d)}\vol(B)
		\tag{R3}
	\end{equation}
	where, this time, $p$ is a polynomial function of degree 2.

	The relations \eqref{R1}--\eqref{R3} have been obtained from estimates of the right hand side of an appropriate integral equation with the sensitivity as unknown. Their precise formulation is based on the numerical fitting of polynomial parameters from a data base. To the best of our knowledge, this point of view has not yet been adopted in literature. Integral operators have been used in \cite{AmmariKangI, AmmariKangII} to develop asymptotic expansions that allow to retrieve informations about the localization and shape of small-volume perturbations in the parameters of both the conductivity and Helmholtz equation.

	\subsection{Numerical verification of the relations (\ref{R1}), (\ref{R2}), and (\ref{R3})}
	\label{sec:conjectures}

	The numerical verification of the above explicit relations has been done in the case where the computational domain is the unit sphere. Let us consider a single spherical perturbation $B = B_{\alpha}(\bfx_0)$ of radius $\alpha > 0$ and center $\bfx_0 \in \Omega$. We study the Gâteaux derivative of the electric field in the direction $\varrho = (0,\textbf{1}_B)$ for sample values of $\bfx_0$ and $\alpha$. For each couple $(\bfx_0, \alpha)$, we compute the tangential trace $\bfE^1 \times \bfn$ on $\Gamma$ where $\bfE^1$ is the solution of the sensitivity equation~\eqref{eq:sensitivity}.

	The physical parameters are the same as in \autoref{sec:numerical_results}. We fix $h = 0.14$ and $N_e = \num{114457}$. In the sequel, we present some illustrations of the three relations, and explain in which way the polynomial functions of relations \eqref{R2} and \eqref{R3} have been obtained.

	\subsubsection{Projection of the perturbation's center}

	Property~\eqref{R1} is illustrated in \autoref{fig:localization_trace}. We report the modulus of the sensitivity $\bfE^1 \times \bfn$ at the boundary in the direction $\varrho = (0,\textbf{1}_{B})$ for a perturbation of the conductivity centered at $\bfx_0 = (-0.85,0,0)$ with radius $\alpha = 0.1$. In order to improve the readability of the image, we use an equirectangular projection and create an image of ratio 2:1. The coordinates $(x,y)$ of each pixel are mapped to $(\theta,\varphi) \in \interval[open right]{0}{2\pi} \times \interval{-\frac{\pi}{2}}{\frac{\pi}{2}}$. Then, the color of the pixel corresponds to the value of the function at coordinates $(\theta,\varphi)$ on the sphere. We observe that the trace $\bfE^1 \times \bfn$ is localized in a neighboorhood of the point $\hat{\bfx} = (-1,0,0)$ which is the projection of the perturbation's center $\bfx_0$ on $\Gamma$. The same localization property is observed for any tested couple $(\bfx_0, \alpha)$.

	\begin{figure}[p]
		\centering
		\begin{minipage}[c]{0.65\textwidth}
			\includegraphics[width=\textwidth]{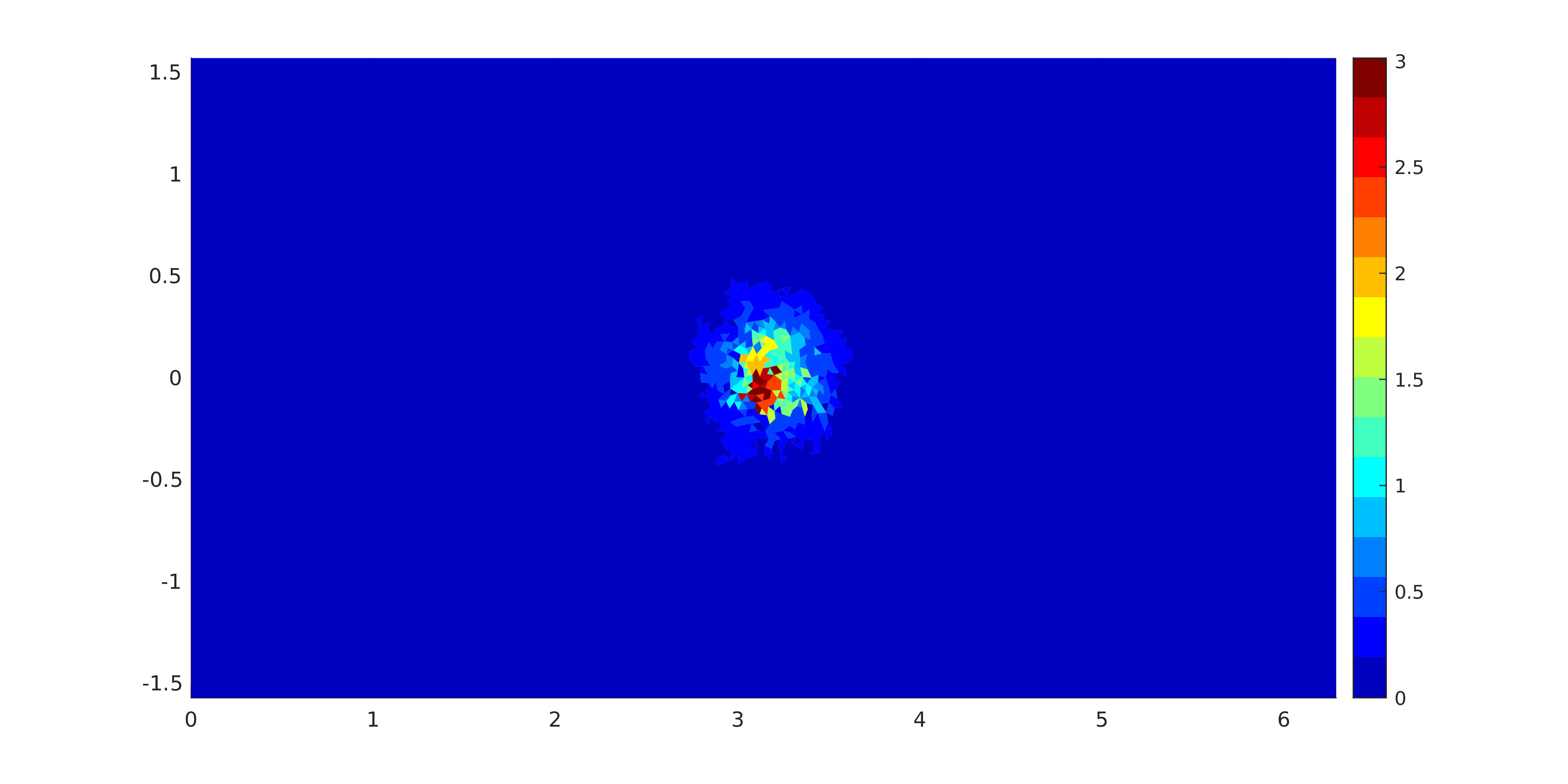}
			\includegraphics[width=\textwidth]{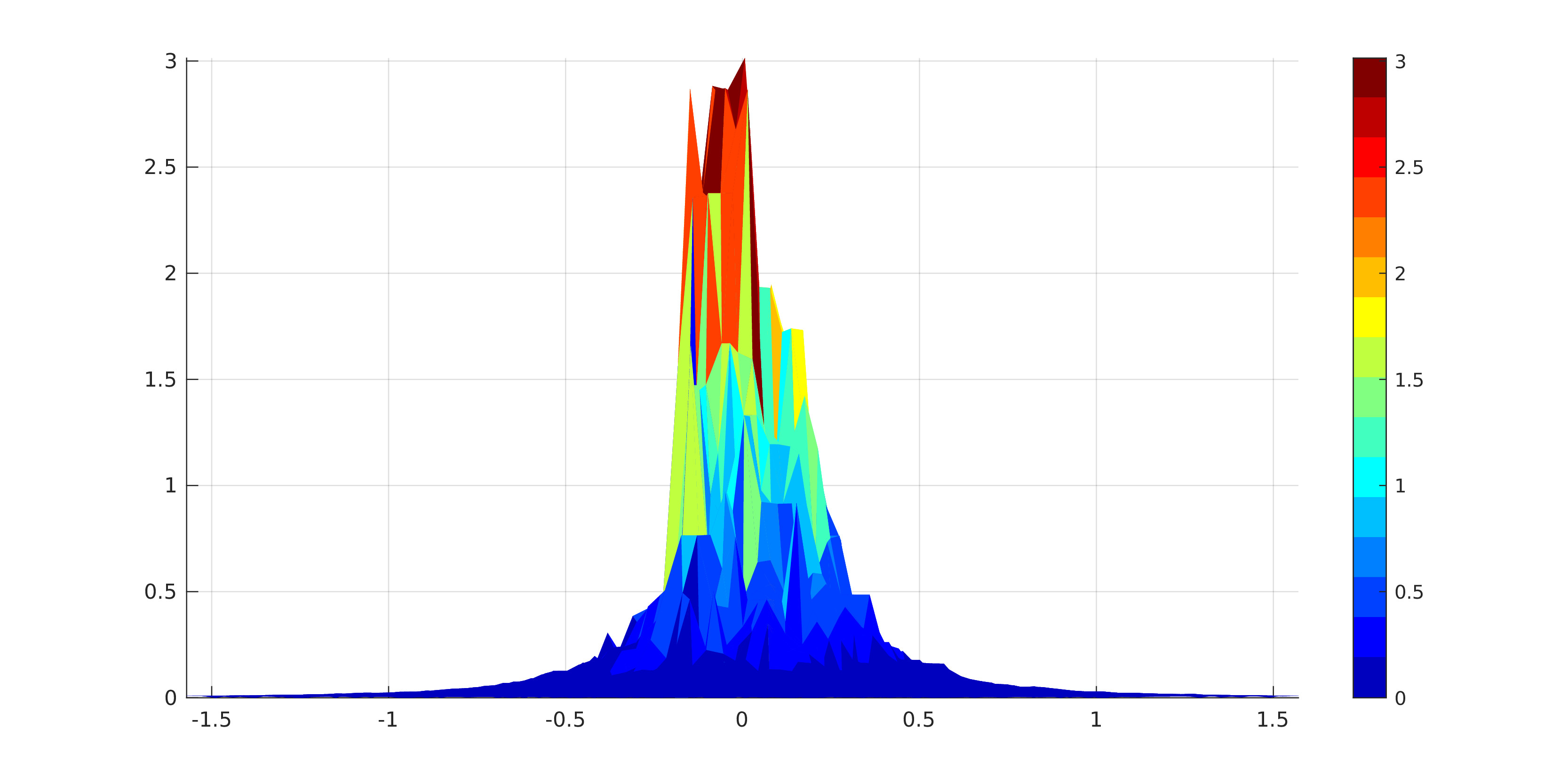}
		\end{minipage}
		\hfill
		\begin{minipage}[c]{0.34\textwidth}
			\caption{Illustration of the relation~\eqref{R1}. Top: modulus of the trace $|\bfE^1 \times \bfn|$ on the unfolded sphere. Bottom: an amplitude peak appears at the point $\hat{\bfx}$ (view from side).}
			\label{fig:localization_trace}
		\end{minipage}
	\end{figure}

	\subsubsection{Depth of the perturbation}

	In order to verify relation~\eqref{R2} and determine numerically the coefficients of the involved polynomial function $p_{\vartheta}$, we fix ${\hat{\bfx}} = (-1,0,0)$ and a radius $\alpha$, and consider different centers $\bfx_0(d) = (1-d) \hat{\bfx}$, $d \in \interval[open]{\alpha}{1}$. For each sample value $d$, we compute numerically the ratio $\area(\Gamma_{\vartheta}(d))/\area(\Gamma)$ from the boundary data $\bfE^1_d \times \bfn$ associated with the direction $\varrho = (0, \textbf{1}_B)$ where $B = B_\alpha(\bfx_0(d))$. We get the distribution function given by \eqref{R2} where the polynomial function $p_\vartheta$ has been chosen to fit the data (see \autoref{fig:depth_areas}, left). Next, we perform the same test for different radii $\alpha$. It turns out that the behavior of the plotted curve is independent from $\alpha$ (\autoref{fig:depth_areas}, right). Therefore, the same polynomial $p_\vartheta$ allows to retrieve the depth $d$ by inverting the relation \eqref{R2} independently from the (unknown) radius $\alpha$.

	\begin{figure}[p]
		\centering
		\includegraphics[width=0.4\textwidth]{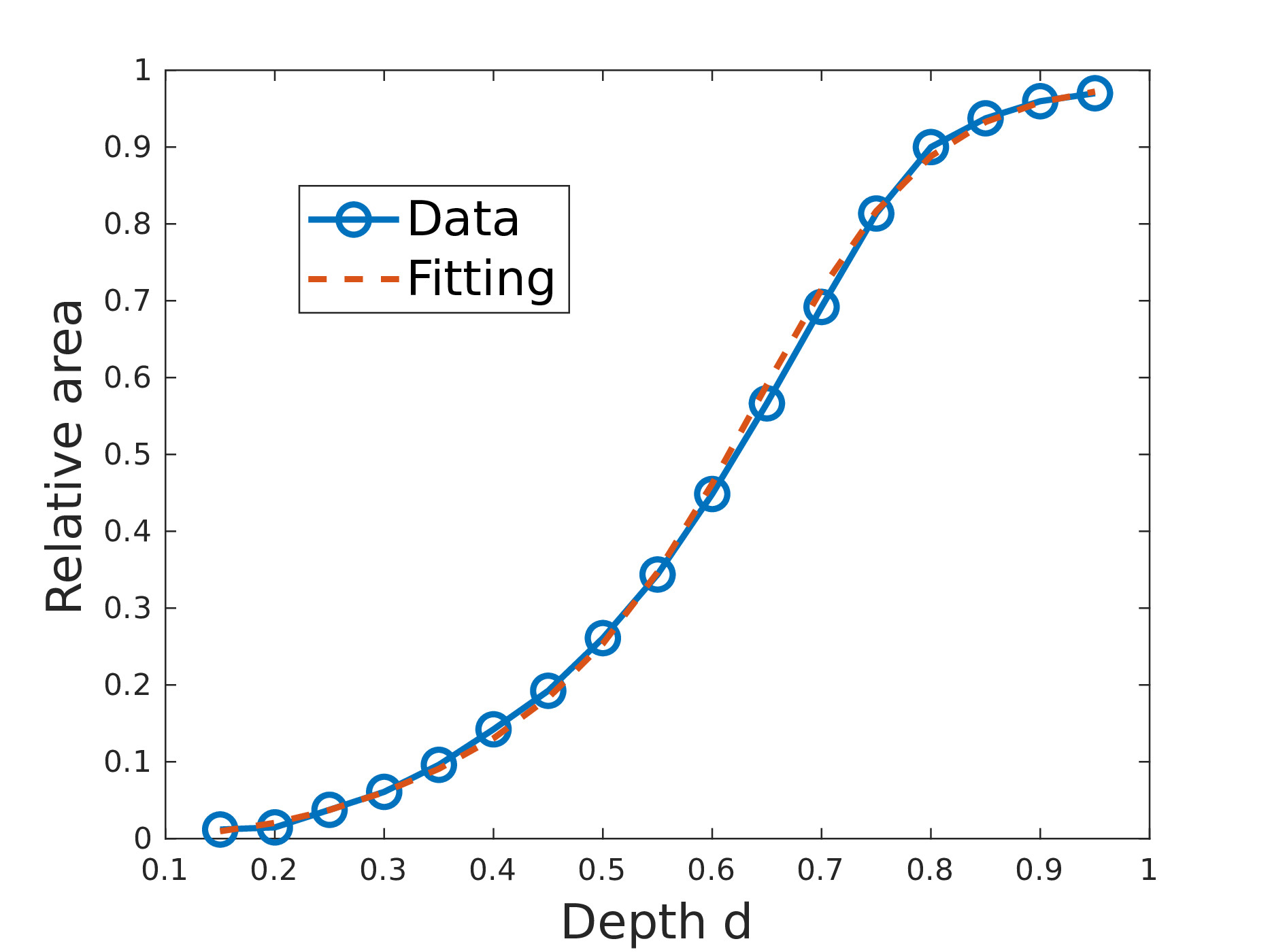}
		\includegraphics[width=0.4\textwidth]{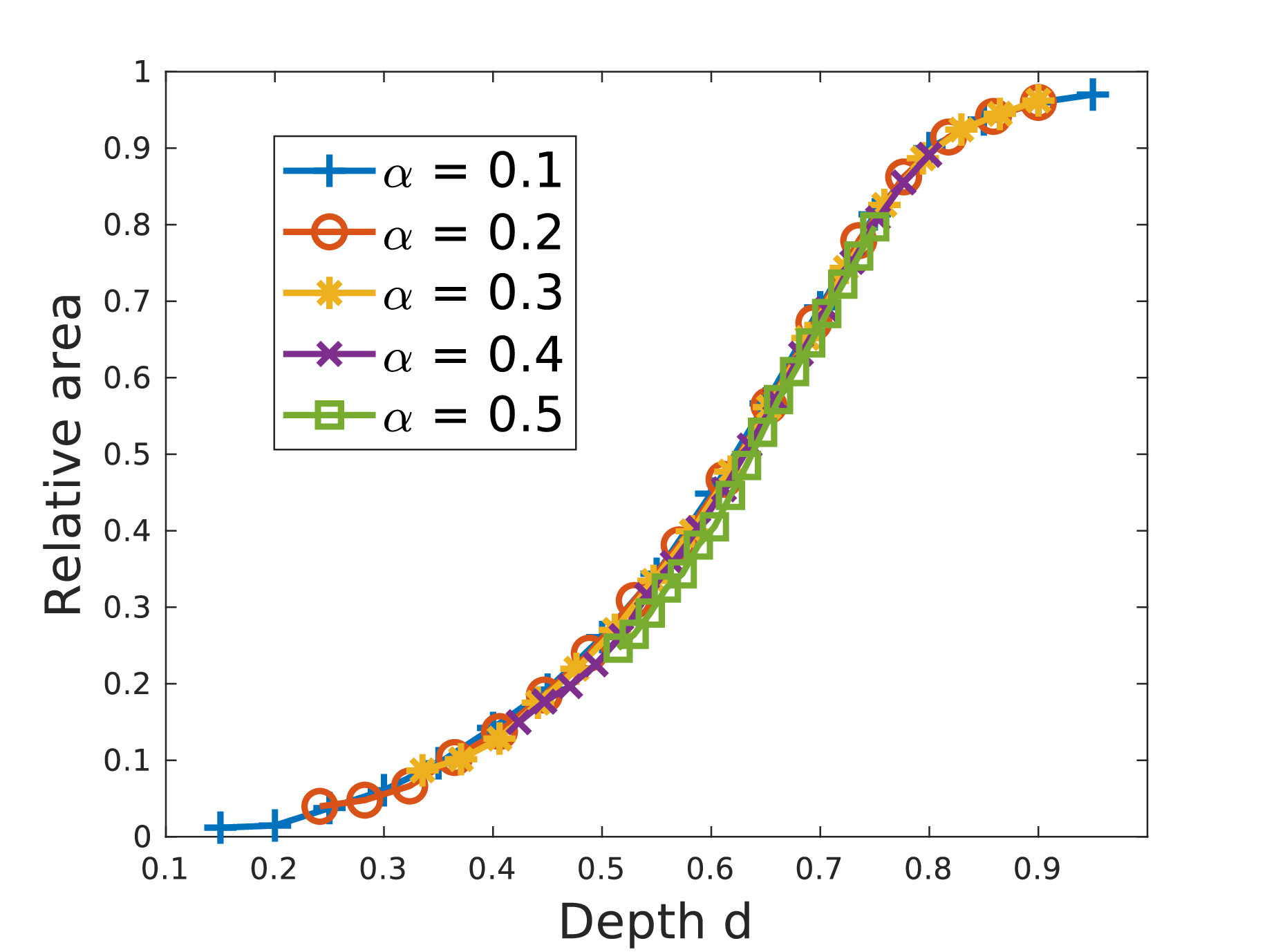}
		\caption{Illustration of relation~\eqref{R2}. Left: fitting of the numerical data to a curve of shape $\dfrac{1}{1 + e^{p_\vartheta(d)}}$, $\alpha = 0.1, \vartheta = 0.2$. Right: illustration of the independance of the relation over the perturbation's size $\alpha$.}
		\label{fig:depth_areas}
	\end{figure}

	\subsubsection{Volume of the perturbation }

	Finally, let us study relation~\eqref{R3}. In \autoref{fig:volume_norm} (left) we plot the $L^2$-norm of the boundary data $\bfE^1 \times \bfn$ in terms of the volume of the perturbation $B$ with fixed center $\bfx_0 = (0,0,0)$ and different values of $\alpha \in \interval[open right]{0}{1}$, where $\bfE^1$ has been computed in the direction $\varrho = (0,\textbf{1}_{B})$. This agrees with the statement \eqref{R3} if we neglect the constant term in the affine relation. However, the linearity constant in \eqref{R3} is likely to depend on the depth $d$. We thus check numerically the value of the constant for different depths $d$. These data fit to a relation of exponential shape,
	\begin{equation}
		\label{eq:lin_const}
		K(d) = e^{p(d)}
	\end{equation}
	with a given polynomial $p$ of degree 2.

	\begin{figure}[p]
		\centering
		\includegraphics[width=0.4\textwidth]{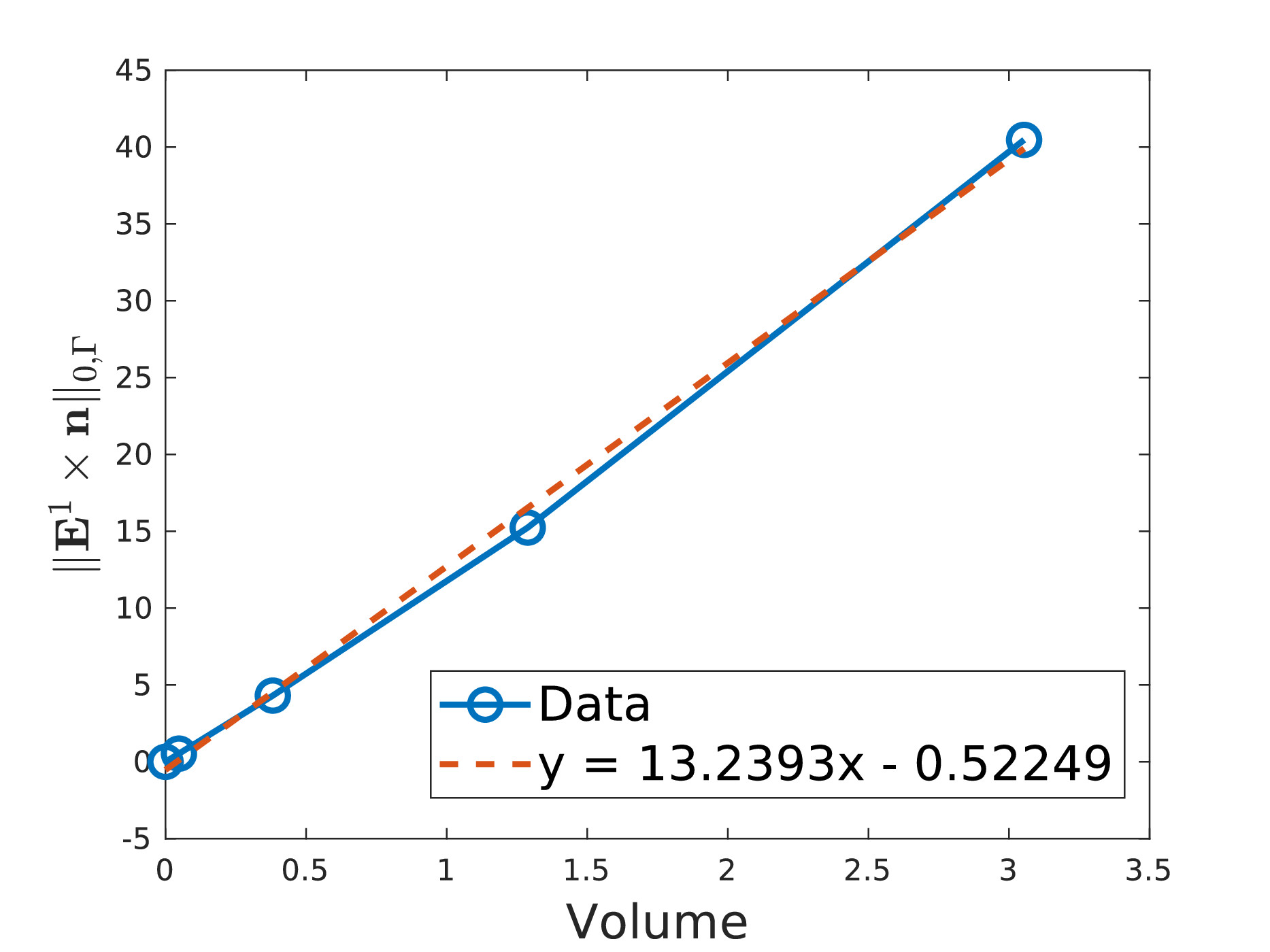}
		\includegraphics[width=0.4\textwidth]{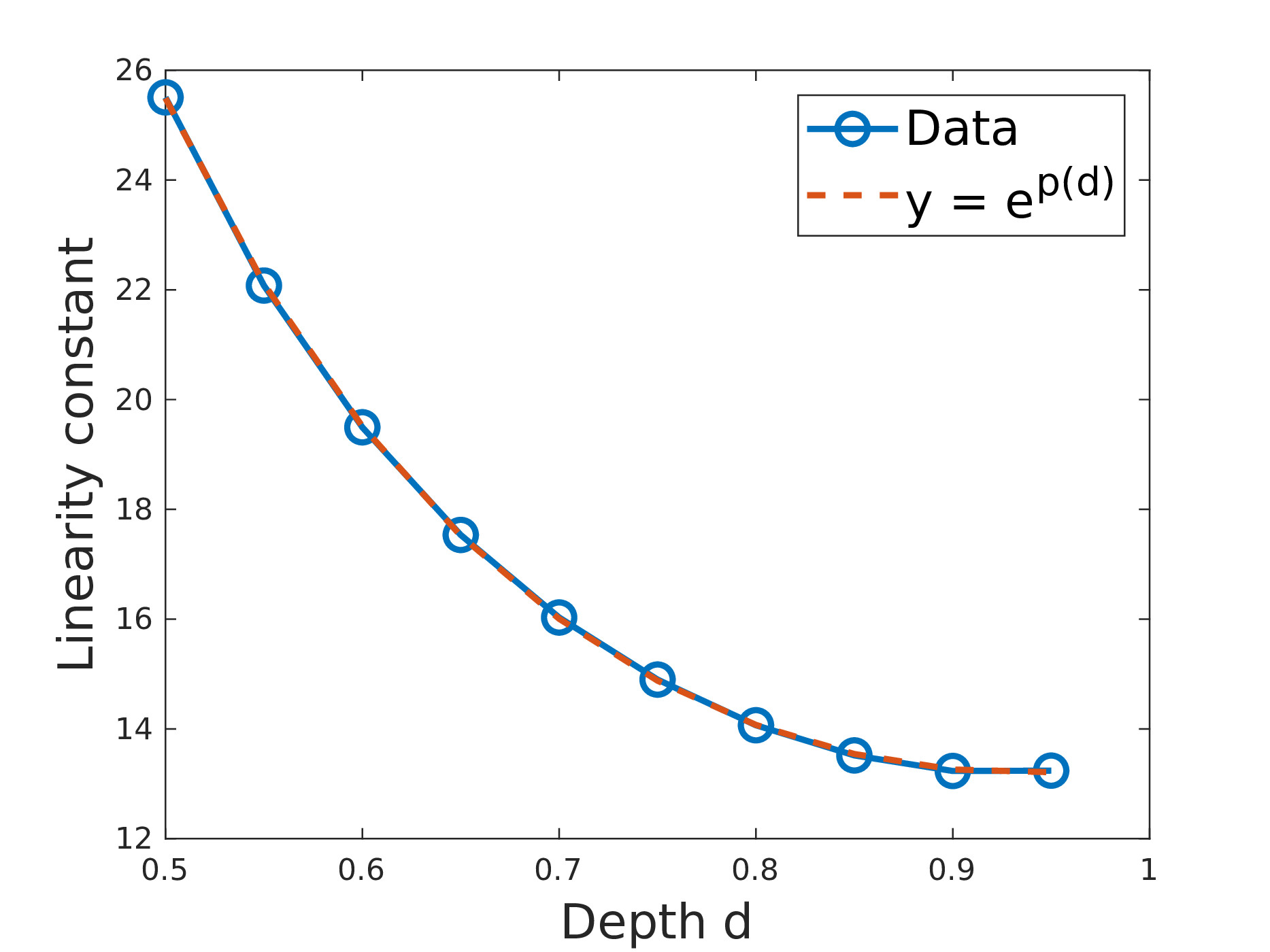}
		\caption{Illustration of the linear relation~\eqref{R3} (left). Evolution of the linearity constant with respect to the depth $d$ (right).}
		\label{fig:volume_norm}
	\end{figure}

	\section{The localization algorithm}
	\label{sec:algo}

	We develop a reconstruction algorithm of interior perturbations from the knowledge of the tangential trace $\bfE^1 \times \bfn$ of the sensitivity (of the electric field). The algorithm is based on the relations~\eqref{R1}, \eqref{R2}, and \eqref{R3} of the former section. Notice that the proposed algorithm could easily be applied to the case where the input data is the tangential trace of the perturbed field according to Taylor expansion~\eqref{eq:taylor} which is valid for small-amplitudes.

	\subsection{Database generation}

	A first step of the inversion algorithm consists in simulating a large number of possible spherical perturbations with same projection $\hat{x}$ and defined by their depth $d$ and their volume. We compute the corresponding boundary data $\bfE^1 \times \bfn$. By varying the inhomogeneity's parameters, we are able to estimate the coefficients of the polynomial functions in \eqref{R2} and \eqref{R3}. These coefficients are stored and will be used in the resolution phase. The database is generated with a given mesh $\mcM_\textrm{data}$. It is important to notice that this preliminary step is required only once for a given computational domain $\Omega$. It is described in \autoref{algo:db}.

	\begin{algorithm}
		\caption{Database generation}
		\label{algo:db}

		\KwIn{physical parameters $\eps$ and $\sigma$ of the background medium, frequency $\omega = \SI{1e6}{\hertz}$. Database mesh $\mcM_\textrm{data}$. Projection $\hat{\bfx}$ and projection direction $\tau$. Direction set $N_m$ of $m$ incident directions $\bfeta$.}
		\KwOut{coefficients of the polynomials $p_\vartheta$ and $p$.}

		\ForEach{incident direction $\bfeta \in N_m$}{
			\textbf{Step 1} \quad Compute the numerical solution $\bfE$ of \eqref{eq:inverse_E}.\;
			\textbf{Step 2} \quad Sample perturbations\;
			\Indp
			\For{$\alpha$ from $0$ to $0.5$}{
				\For{$d$ from $\alpha$ (excluded) to $1$}{
					1. Compute $\bfx_0 = \hat{\bfx} + d*\tau$.\;
					2. Compute the numerical solution $\bfE^1$ of the sensitivity equation~\eqref{eq:sensitivity} in the direction $(0,\textbf{1}_B)$ where $B = B_\alpha(\bfx_0)$.\;
					3. Compute the tangential trace $\bfE^1 \times \bfn$ on $\Gamma$.\;
				}
			}
			\Indm
			\textbf{Step 3} \quad Fit the coefficients of polynomials $p_\vartheta$ and $p$ to relations \eqref{R2} and \eqref{R3}.\;
		}
	\end{algorithm}

	\subsection{The inversion procedure}

	The algorithm is presented hereafter in the case of one single perturbation in the conductivity. The aim is to retrieve the following parameters from given (synthetic) boundary data $\bfE^1 \times \bfn$:
	\begin{itemize}
		\item $\hat{\bfx}$, the projection on the sphere of $\bfx_0$, the center of the interior perturbation,
		\item $d = |\bfx_0 - \hat{\bfx}|$, the depth of the center $\bfx_0$,
		\item the volume $v_B$ of the perturbation which yields the radius $\alpha$ in the case of a spherical perturbation.
	\end{itemize}

	\begin{figure}
		\centering
		\includegraphics[width=\textwidth]{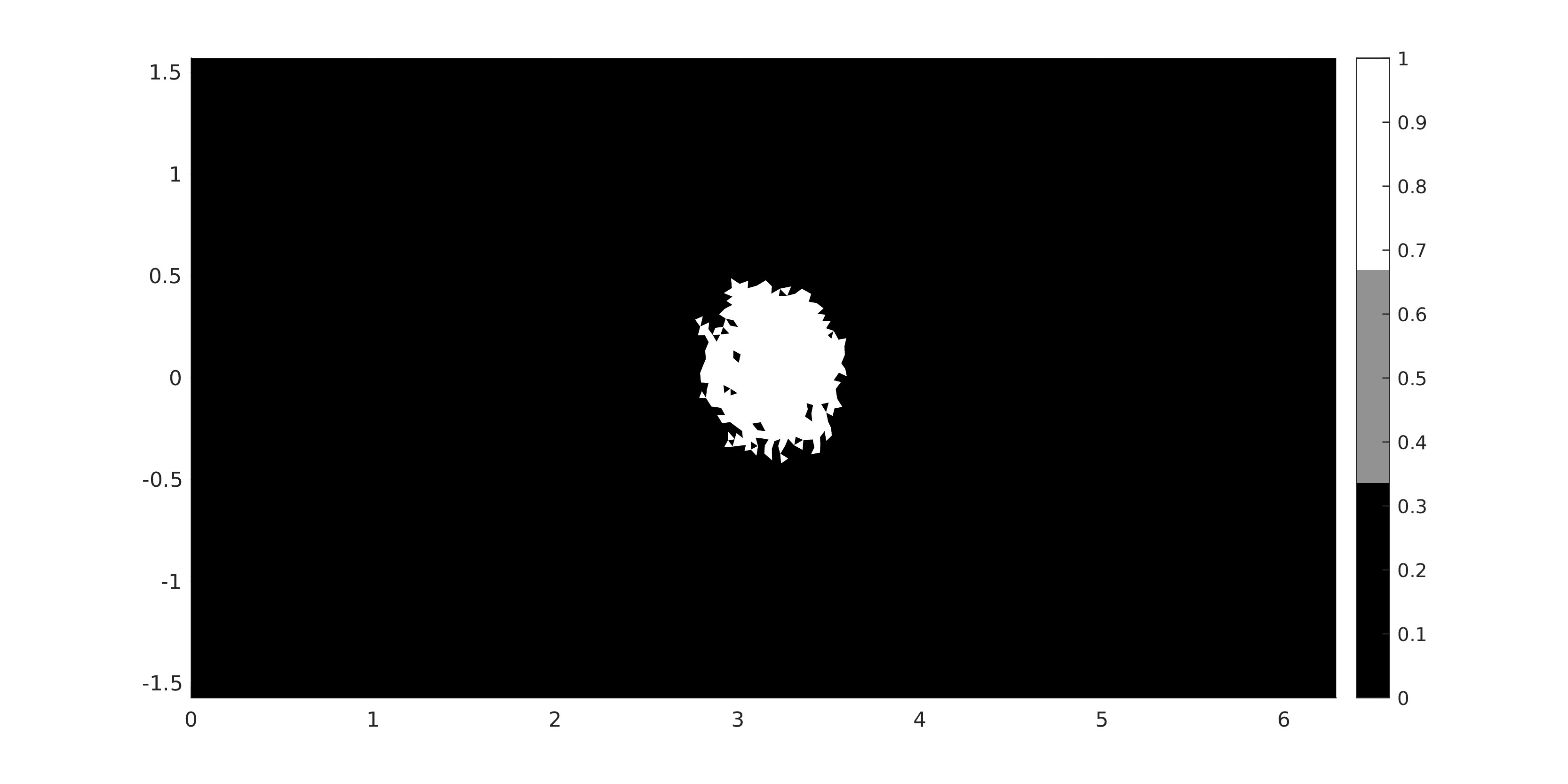}
		\caption{Example of a thresholded modulus of the trace $\bfE^1 \times \bfn$. The surfacic perturbations of largest amplitudes are shown in white.}
		\label{fig:algo_thresholded_trace}
	\end{figure}

	\begin{algorithm}
		\caption{Localization}
		\label{algo:loc}

		\KwIn{discrete (synthetic) boundary data $\bfE^1 \times \bfn$}
		\KwOut{projection $\hat{\bfx}$, depth $d$ and volume $v_B$ of the perturbation}
		\KwParam{threshold $\vartheta = 0.2$}

		\textbf{Step 1} \quad \textbf{projection $\hat{\bfx}$ on $\Gamma$ of the perturbation's center} \eqref{R1}\;
		\Indp
		1. Convert the thresholded modulus of the data (see \eqref{defGamma}) to a black-and-white image by threshold $\beta$ on the unfolded sphere (see \autoref{fig:algo_thresholded_trace}). The threshold $\beta$ is chosen relatively to the maximum value of the modulus $|(\bfE^1 \times \bfn)|$ (half of this value in the examples).\;
		2. Compute the convex hull $\mcC$ of the white pixels by Graham's algorithm \cite{Graham72} (or other).\;
		3. Let $\hat{\bfx}$ the isobarycenter of $\mcC$.\;
		\Indm
		\textbf{Step 2} \quad \textbf{depth $d$ of the perturbation} \eqref{R2}\;
		\Indp
		1. Compute the ratio $r = \area(\Gamma_{\vartheta}(d))/\area(\Gamma)$ from the data.\;
		2. Find $d$ such that $1/(1+e^{p_\vartheta(d)}) = r$.\;
		\Indm
		\textbf{Step 3} \quad \textbf{volume $v_B$ of the perturbation} \eqref{R3}\;
		\Indp
		1. Compute the $L^2$-norm $\norm{\bfE^1 \times \bfn}{0,\Gamma}$ of the boundary data.\;
		2. Compute $K = e^{p(d)}$ from \eqref{eq:lin_const} and $d$ from Step 2.\;
		3. Let $v_B = \norm{\bfE^1 \times \bfn}{0,\Gamma}/K$.\;
	\end{algorithm}

	\subsection{Numerical simulations}
	\subsubsection{Generation of synthetic boundary data}

	In the absence of measurements, we generate discrete synthetic boundary data in the following way.

	\begin{itemize}
		\item Fix a perturbation $B = B_\alpha(\bfx_0)$ with given center $\bfx_0$ and radius $\alpha$.
		\item For any direction $\bfeta$ in a given set $N_m$ of $m$ directions
			\begin{itemize}
				\item Compute the boundary data $\bfE^1 \times \bfn$ where
				$\bfE^1$ is the sensitivity in the direction $\varrho = (0,\textbf{1}_B)$ for the source term $\bfE_\eta$.
				\item Add some noise to the data (\autoref{ss:noisydata} only).
			\end{itemize}
		\item Take $\bfa = (\bfE^1 \times \bfn) + \textrm{noise}$ as input data for the inversion algorithm.
	\end{itemize}
	The following sets of directions are used in the sequel
	\begin{align*}
		N_1 &= \collection{(0,1,0)}, \\
		N_6 &= N_1 \cup \collection{(1,0,0), (0,0,1), (-1,0,0), (0,-1,0), (0,0,-1)}, \\
		N_{14} &= \begin{multlined}[t]
			N_6 \cup \collection{(1,1,1), (-1,1,1), (1,-1,1), (1,1,-1), \\
			(-1,-1,1), (-1,1,-1), (1,-1,-1), (-1,-1,-1)}.
		\end{multlined}
	\end{align*}

	In order to avoid an inverse crime (in the sense of \cite[p.~133]{CK98}), the boundary data are computed on a tetrahedral mesh $\mcM_\textrm{inv}$ of size $h = 0.12$ ($N_e = \num{167402}$ edges) which is different from the mesh used to generate the database. Notice also that we focus on perturbations of the conductivity parameter. We refer to \autoref{ss:twoparam} where both parameters,
	$\eps$ and $\sigma$, undergo a perturbation.

	\subsubsection{Spherical perturbation}
	\label{ss:spherical}

	We first apply \autoref{algo:loc} in a homogeneous background medium containing a spherical perturbation in the conductivity. We keep the physical settings defined in \autoref{sec:numerical_results}. In the case of multiple directions, we apply the algorithm on each direction $\bfeta \in N_m$ and compute the mean of the results.

	We first test our algorithm with a spherical perturbation centered at $\bfx_0 = (-0.7,0,0)$ and of radius $0.2$. The parameters to retrieve are then $\hat{\bfx} = (-1,0,0)$ (or, in spherical coordinates, $\hat{\bfx} = (\pi,0)$), $d = 0.3$ and $\alpha = 0.2$. We report the results in \autoref{tab:inverse_homogeneous_x}. The approximations of $\hat{\bfx}$, $d$ and $\alpha$ are respectively denoted by $\hat{\bfx}_h$, $d_h$ and $\alpha_h$. We choose the Euclidian norm of the spherical coordinates of the points $\hat{\bfx}$ and $\hat{\bfx}_h$ to compute the projection error.

	\begin{table}[hbt]
		\caption{Numerical resolution of the inverse problem in a homogeneous medium. One spherical perturbation in the conductivity, of center $\bfx_0 = (-0.7,0,0)$ (i.e. $\hat{\bfx} = (\pi,0)$) and radius $\alpha = 0.2$.}
		\label{tab:inverse_homogeneous_x}
		\centering
		\begin{tabular}{cccc}
			\toprule
			Incidences & $N_1$ & $N_6$ & $N_{14}$ \\
			\midrule
			$\hat{\bfx}_h$ & (3.156,0.018) & (3.152,0.015) & (3.151,0.009) \\
			$\frac{|\hat{\bfx} - \hat{\bfx}_h|}{|\hat{\bfx}|}$ & 7.502e-03 & 5.713e-03 & 4.173e-03 \\
			$d_h$ & 0.316 & 0.311 & 0.313 \\
			$\frac{|d - d_h|}{|d|}$ & 5.459e-02 & 3.728e-02 & 4.193e-02 \\
			$\alpha_h$ & 0.214 & 0.214 & 0.210 \\
			$\frac{|\alpha - \alpha_h|}{|\alpha|}$ & 6.943e-02 & 7.158e-02 & 5.037e-02 \\
			\bottomrule
		\end{tabular}
	\end{table}

	We retrieve the projection center $\hat{\bfx}$ with a very good accuracy (less than \pc{1} error). The approximation error on $\hat{\bfx}$ decreases with respect to the number of incident waves. The other perturbation's characteristics $d$ and $\alpha$ are well approximated, too (about \pc{4} to \pc{7} error). Their approximation does not really depend on the number of waves. In order to keep a reasonable number of computations (resp. measurements in the context of biomedial applications that we have in mind), we decide to work from now on with the set $N_6$ of 6 different incident waves (unless specified otherwise).

	We now test a spherical perturbation which is centered at a different point. We keep the parameters $d = 0.3$ and $\alpha = 0.2$, and consider the center $\bfx_0 = (0,-0.7,0)$ (which corresponds to $\hat{\bfx} = (0,-1,0)$. In \autoref{tab:inverse_homogeneous_y}, we compare the errors obtained with the two configurations. The relative errors of the two approximations are comparable. Changing the position of the perturbation does not affect the quality of the approximation. From now on, we thus will consider perturbations centered on the $x$-axis (unless indicated otherwise).

	\begin{table}[hbt]
		\caption{Comparison of the results for different centers. Relative errors for a spherical perturbation of radius $\alpha = 0.2$ in the conductivity.}
		\label{tab:inverse_homogeneous_y}
		\centering

		\begin{tabular}{cccc}
			\toprule
			Perturbation's center & $\frac{|\hat{\bfx} - \hat{\bfx}_h|}{|\hat{\bfx}|}$ & $\frac{|d - d_h|}{|d|}$ & $\frac{|\alpha - \alpha_h|}{|\alpha|}$ \\
			\midrule
			$\bfx_0 = (-0.7,0,0)$ & 5.713e-03 & 3.728e-02 & 7.158e-02 \\
			$\bfx_0 = (0,-0.7,0)$ & 5.551e-03 & 5.528e-02 & 9.906e-02 \\
			\bottomrule
		\end{tabular}
	\end{table}

	We next apply the algorithm for different depths and volums and report the errors in \autoref{tab:inverse_homogeneous_dr}. It may be stated that the accuracy of the reconstruction does not depend significantly on the depth or volume of the perturbation.

	\begin{table}[hbt]
		\caption{Comparison of the results for different depths and radii. Relative errors for a spherical perturbation in the conductivity.}
		\label{tab:inverse_homogeneous_dr}
		\centering

		\begin{tabular}{ccccc}
			\toprule
			Perturbation's center & $r$ & $\frac{|\hat{\bfx} - \hat{\bfx}_h|}{|\hat{\bfx}|}$ & $\frac{|d - d_h|}{|d|}$ & $\frac{|\alpha - \alpha_h|}{|\alpha|}$ \\
			\midrule
			$\bfx_0 = (-0.7,0,0)$ & $0.2$ & 5.713e-03 & 3.728e-02 & 7.158e-02 \\
			\multirow{2}{*}{$\bfx_0 = (-0.4,0,0)$} & $0.2$ & 9.658e-03 & 6.661e-03 & 1.354e-02 \\
			& $0.4$ & 1.110e-02 & 9.791e-03 & 1.057e-02 \\
			\bottomrule
		\end{tabular}
	\end{table}

	\subsubsection{Noisy data}
	\label{ss:noisydata}

	An important feature in numerical reconstruction is noise robustness. Noisy synthetic data are generate in the following way. Let ${(z_k)}_k = {(x_k + iy_k)}_k$ be the degrees of freedom of the sensitivity $\bfE^1$. We add noise independently in the real and imaginary parts. Let ${(n_k)}_k$ be a vector of real random numbers following a normal law $\mcN(0,\sigma)$. Then, we generate additive noise in the real part with the following formula:
	\[x^\text{noise}_k = x_k + \frac{M_x - m_x}{2}n_k + \frac{m_x + M_x}{2},\]
	where $m_x = \min\limits_k x_k$ and $M_x = \max\limits_k x_k$. The imaginary part is jittered in the same way.

	In \autoref{tab:inverse_noise}, we compare the results obtained by our algorithm using noisy or non-noisy data (with a single incident wave of direction $(0,1,0)$). As expected, noise affects the reconstruction. A numerical observation of the noisy boundary data allows to get deeper insight on the effect of noise and to propose a strategy for noise reduction. Indeed, in a normally distributed noise vector, a few coefficients may have large values. Consequently, there can be a small number of points on the boundary located far away from the projection $\hat{x}$, and at which the modulus $|\bfE^1 \times \bfn|$ is greater or equal than at the points around the projection $\hat{\bfx}$ (see \autoref{fig:trace_noise}). This causes problems to the first step of \autoref{algo:loc}. A solution is to apply a post-processing which consists in identifying outliers and deciding whether they should be retained or rejected. As before, we retrieve the points where the largest values of the modulus are reached. An outlier detection is applied to remove only those points which are not in the neighborhood of the principal peak (i.e. centered at $\hat{\bfx}$). For instance, the deviation around the median can be used (see~\cite{Leys13}). By applying this simple algorithm twice, we reduce the number of outliers significantly. In \autoref{tab:inverse_noise_denoised}, we report the reconstruction results obtained from the noisy data that have undergone the outlier detection, compared to resultas from non-noisy data. Up to \pc{2} of noise, the results are not affected by the noise, and the projection $\hat{x}$ is retrieved with a similar precision even for \pc{5} of noise. At \pc{10} of noise, the projection is found with \pc{12} error, but the precision of the other parameters is no longer significant. This first study of noisy data indicates that the detection of outliers is an interesting and promising approach, and different noise reduction methods could be tested to improve the results. This was however beyond the scope of this paper.

	\begin{figure}[hbt]
		\centering
		\includegraphics[width=\textwidth]{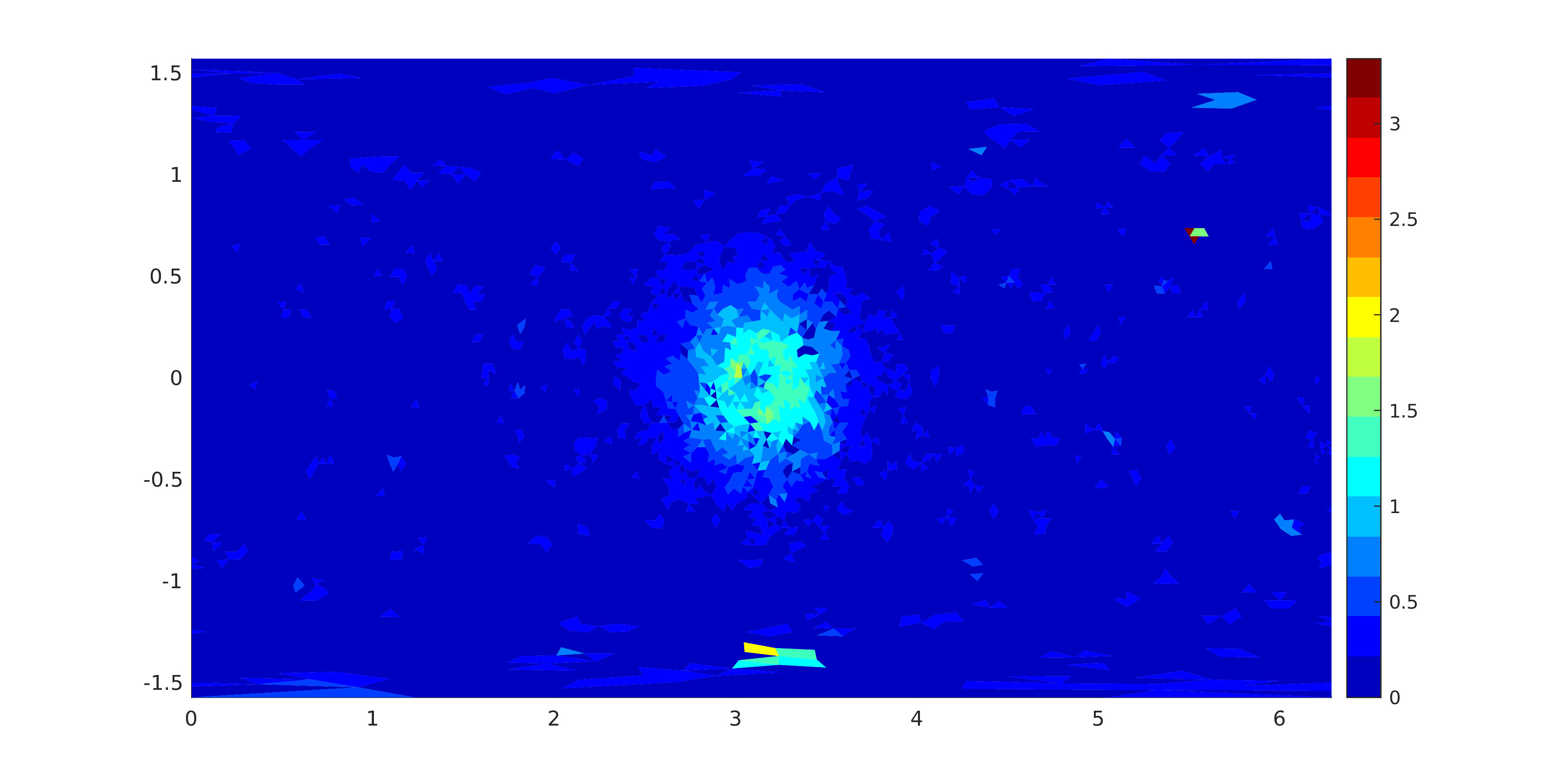}
		\caption{Example of the modulus of a noised trace (here with \pc{2} of noise). Some points with a value bigger than the central peak can be observed.}
		\label{fig:trace_noise}
	\end{figure}

	\begin{table}[hbt]
		\caption{Reconstruction of one spherical perturbation centered at $\bfx_0 = (-0.7,0,0)$ of radius $r = 0.2$ from non-noisy or noisy data.}
		\label{tab:inverse_noise}
		\centering

		\begin{tabular}{cccc}
			\toprule
			& $\frac{|\hat{\bfx} - \hat{\bfx}_h|}{|\hat{\bfx}|}$ & $\frac{|d - d_h|}{|d|}$ & $\frac{|\alpha - \alpha_h|}{|\alpha|}$ \\
			\midrule
			Non-noisy data & 7.502e-03 & 5.459e-02 & 7.474e-02 \\
			Noisy data ($\sigma = \pc{2}$) & 2.835e-01 & 4.945e-01 & 1.243e-01 \\
			\bottomrule
		\end{tabular}
	\end{table}

	\begin{table}[hbt]
		\caption{Reconstruction of one spherical perturbation centered at $\bfx_0 = (-0.7,0,0)$ of radius $r = 0.2$ from non-noisy or noisy data (after the detection
		of outliers).}
		\label{tab:inverse_noise_denoised}
		\centering

		\begin{tabular}{cccc}
			\toprule
			& $\frac{|\hat{\bfx} - \hat{\bfx}_h|}{|\hat{\bfx}|}$ & $\frac{|d - d_h|}{|d|}$ & $\frac{|\alpha - \alpha_h|}{|\alpha|}$ \\
			\midrule
			Non-noisy data & 7.502e-03 & 5.459e-02 & 7.474e-02 \\
			Denoised data ($\sigma = \pc{1}$) & 3.942e-03 & 4.587e-02 & 6.175e-02 \\
			Denoised data ($\sigma = \pc{2}$) & 4.826e-03 & 4.541e-02 & 6.286e-02 \\
			Denoised data ($\sigma = \pc{5}$) & 4.642e-03 & 2.244e-01 & 1.711e-01 \\
			Denoised data ($\sigma = \pc{10}$) & 1.241e-02 & 9.144e-01 & 5.082e-01 \\
			\bottomrule
		\end{tabular}
	\end{table}

	\subsubsection{Simultaneous reconstruction of both parameters}
	\label{ss:twoparam}

	We test our algorithm for determining a perturbation in the conductivity and the permittivity, located in the same sphere of center $\bfx_0 = (-0.7,0,0)$ and of radius $\alpha = 0.2$. This happens for instance when a stroke occurs in the brain. In \autoref{tab:inverse_two}, we compare the results with the ones obtained where only the conductivity is perturbed. It has to be noted that only perturbations in the conductivity have been used to generate the database. The good accuracy on the approximated center and depth (i.e. the localization error) is preserved. We observe a slight loss of precision for the radius: for an exact radius of $\alpha = 0.2$, we find $\alpha_h = 0.242$ instead of $\alpha_h = 0.214$ if only the conductivity undergoes a perturbation. This is due to the fact that the $L^2$-norm of the tangential trace (used in~\eqref{R3}) is bigger if both parameters are perturbed. A solution could be to construct a database relative to perturbations on both parameters. Nevertheless, the present study shows that the reconstruction is satisfactory even if the database does not take into account information about which parameter is perturbed.

	\begin{table}[hbt]
		\caption{Perturbation in both the conductivity and the permittivity, centered at $\bfx_0 = (-0.7,0,0)$ of radius $\alpha = 0.2$.}
		\label{tab:inverse_two}
		\centering

		\begin{tabular}{cccc}
			\toprule
			Perturbed parameter(s) & $\frac{|\hat{\bfx} - \hat{\bfx}_h|}{|\hat{\bfx}|}$ & $\frac{|d - d_h|}{|d|}$ & $\frac{|\alpha - \alpha_h|}{|\alpha|}$ \\
			\midrule
			Conductivity & 5.713e-03 & 3.728e-02 & 7.158e-02 \\
			Conductivity \& permittivity & 5.666e-03 & 3.813e-02 & 2.080e-01 \\
			\bottomrule
		\end{tabular}
	\end{table}

	\subsubsection{Reconstruction of two perturbations}

	As stated in \autoref{prop:disjoint}, in the case of $n > 1$ disjoint perturbations, the total sensitivity can be separated into $n$ sensitivities, each representing one perturbation. We use this property to handle the configuration with two or more inhomogeneities. To this end, we proceed as follows. The entry data is the trace $\bfE^1 \times \bfn$ which is containing $n$ surfacic perturbations (see \autoref{fig:sensitivity_xy}). The first step is to detect the different amplitude clusterings. This can be achieved by applying a detection algorithm such as DBSCAN which yields the connected components of the thresholded data. DBSCAN is used in data mining (see~\cite{EsterKriegelSanderXu96}) and has the advantage that the a priori knowledge of $n$ is not needed. Next, for each connected component, we build an artificial piecewise trace of the sensitivity: it equals the values of the original trace on the region of the connected component, and is zero otherwise. Finally, we use this new trace as the entry of our localization procedure. In \autoref{tab:inverse_disjoint}, we show the results for two disjoint spherical perturbations in the conductivity. One is centered at $\bfx_{0,1} = (-0.18,0.41,0)$ with radius $\alpha_1 = 0.35$ and the other at $\bfx_{0,2} = (0,-0.7,0)$ with $\alpha_2 = 0.2$. Both perturbations are very well localized, and the volume of the biggest one is better approximated. This procedure will work if the perturbations give raise to well separated projections on the boundary.

	\begin{table}[hbt]
		\caption{Reconstruction of two disjoint perturbations in the conductivity.}
		\label{tab:inverse_disjoint}
		\centering

		\begin{tabular}{cccc}
			\toprule
			Perturbation's center & $\frac{|\hat{\bfx} - \hat{\bfx}_h|}{|\hat{\bfx}|}$ & $\frac{|d - d_h|}{|d|}$ & $\frac{|\alpha - \alpha_h|}{|\alpha|}$ \\
			\midrule
			$\bfx_{0,1} = (-0.18,0.41,0)$ & 8.556e-03 & 2.829e-04 & 3.599e-02 \\
			$\bfx_{0,2} = (0,-0.7,0)$ & 2.031e-03 & 9.596e-02 & 1.334e-01 \\
			\bottomrule
		\end{tabular}
	\end{table}

	\subsubsection{Ellipsoidal perturbation}
	\label{ss:ellipsoidal}

	In real life applications, the shape of the perturbation is in general not known. In this subsection, we want to retrieve an ellipsoidal perturbation. It is centered at $\bfx_0 = (-0.4,0,0)$, of $x$-radius 0.2, $y$-radius 0.4 and $z$-radius 0.2 which yields a volume $v = \scnum{6.702e-02}$. We recall that only spherical perturbations have been considered to generate the database. We report the results in \autoref{tab:inverse_ellips} and compare them with the results obtained in the case of a sphere with same volume. This reconstruction is illustrated in \autoref{fig:inv_ellips}. The ellipsoidal shape does not really affect the precision of the approximations. This may indicate that the polynomials in relations \eqref{R2} and \eqref{R3} are independant from the shape of perturbation.

	\begin{table}[hbt]
		\caption{Reconstruction of an ellipsoidal perturbation in the conductivity, centered at $\bfx_0 =(-0.4,0,0)$ of volume $v=\scnum{6.702e-02}$.}
		\label{tab:inverse_ellips}
		\centering

		\begin{tabular}{cccc}
			\toprule
			Perturbation's shape & $\frac{|\hat{\bfx} - \hat{\bfx}_h|}{|\hat{\bfx}|}$ & $\frac{|d - d_h|}{|d|}$ & $\frac{|v - v_h|}{|v|}$ \\
			\midrule
			Sphere & 9.482e-03 & 8.522e-03 & 1.344e-02 \\
			Ellipsoid & 2.725e-04 & 2.993e-02 & 3.742e-01 \\
			\bottomrule
		\end{tabular}
	\end{table}

	\begin{figure}[hbt]
		\centering
		\includegraphics[width=0.49\textwidth]{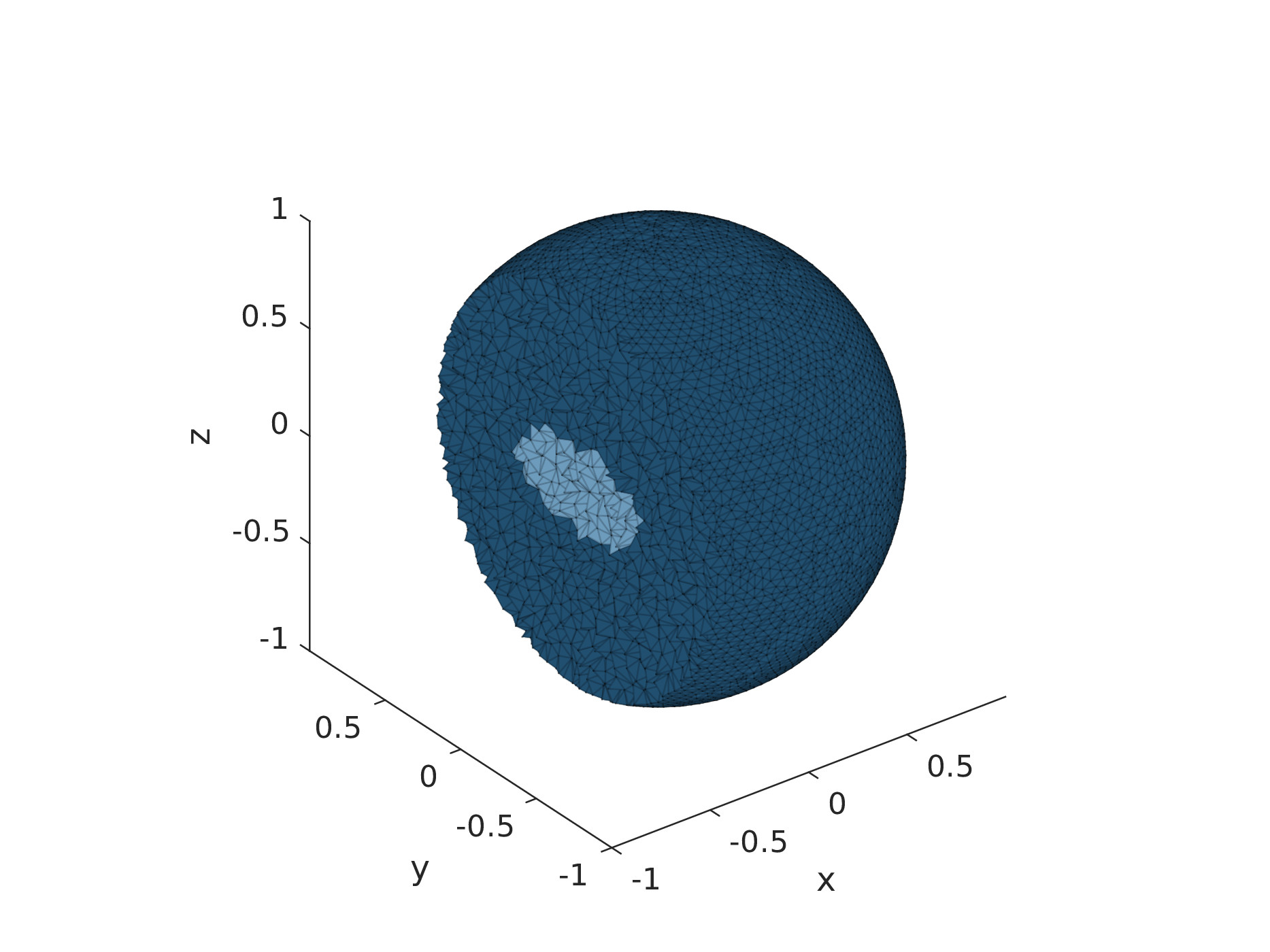}
		\hfill
		\includegraphics[width=0.49\textwidth]{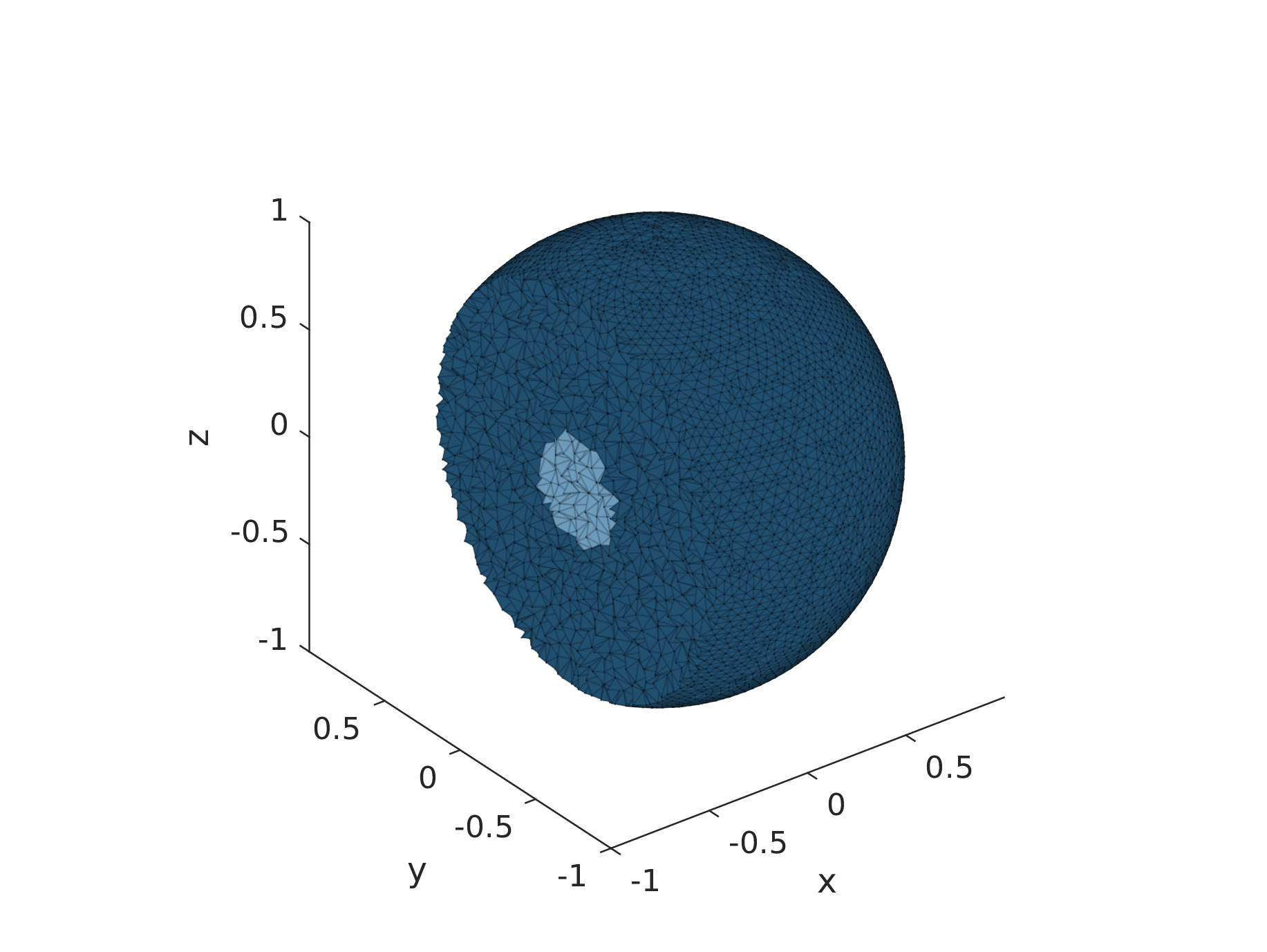}
		\caption{Localization of a perturbation of ellipsoidal shape. Left: induced ellipsoidal perturbation. Right: retrieved perturbation (ball of equivalent volume).}
		\label{fig:inv_ellips}
	\end{figure}

	\subsubsection{Three-layer head spherical model}

	In the biomedical applications we have in mind, e.g. for the diagnostic of strokes, it is important to take into account the heterogeneity of the medium. A classical spherical head model is commonly used in the literature. This model is built of three concentric spheres representing brain, skull and scalp (see \autoref{fig:threelayers}). The parameters of this model are described in \autoref{tab:threelayers_def}. We simulate a spherical inhomogeneity in the conductivity of the brain layer, centered at $\bfx_0 = (-0.57,0,0)$ of radius $\alpha = 0.2$. The errors are reported in \autoref{tab:inverse_threelayers} and compared to the reconstruction results of the same perturbation in a homogeneous background. The approximations are of the same order. The theory has been developped in a case of a homogeneous background but the localization algorithm still offers good results in more realistic configurations. We emphasize that the polynomials in the database were generated with the piecewise constant background conductivity.

	\begin{figure}[hbt]
		\centering
		\includegraphics[width=0.32\textwidth]{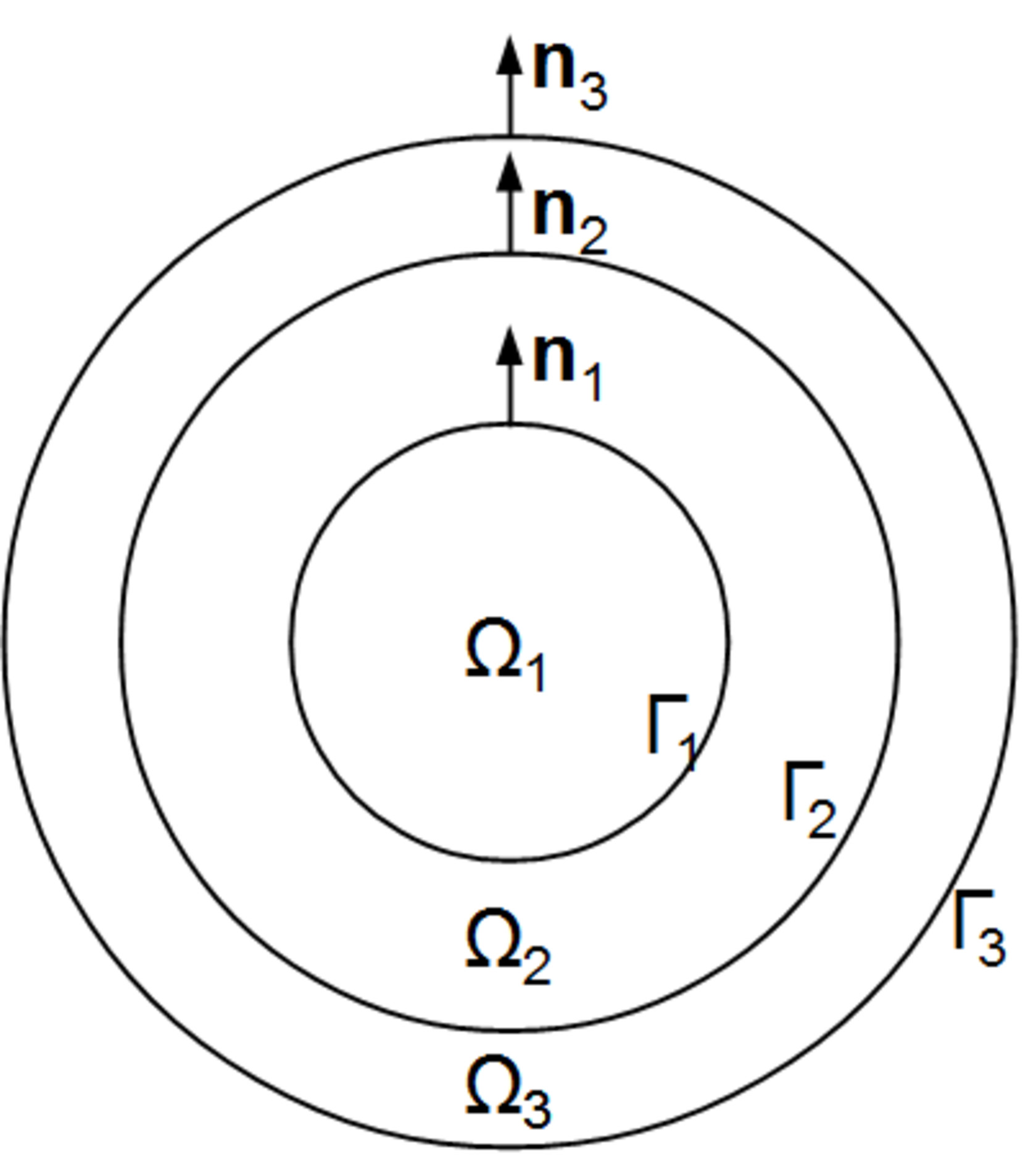}
		\caption{Three-layer head spherical model}
		\label{fig:threelayers}
	\end{figure}

	\begin{table}[hbt]
		\caption{Parameters of the three-layers head model}
		\label{tab:threelayers_def}
		\centering

		\begin{tabular}{cccc}
			\toprule
			Layer & Brain ($\Omega_1$) & Skull ($\Omega_2$) & Scalp ($\Omega_3$) \\
			\midrule
			Definition & $B_{0.87}(0)$ & $B_{0.92}(0) \setminus \Omega_1$ & $\Omega \setminus (\Omega_1 \cup \Omega_2)$ \\
			Permittivity (\si{\farad\per\meter}) & $\num{8.854e-10}$ & $\num{3.542e-10}$ & $\num{8.854e-11}$ \\
			Conductivity (\si{\siemens\per\meter}) & $0.33$ & $0.04$ & $0.33$ \\
			\bottomrule
		\end{tabular}
	\end{table}

	\begin{table}[hbt]
		\caption{Three-layer spherical head model. Reconstruction of a spherical perturbation in the conductivity centered at $\bfx_0 = (-0.57,0,0)$ of radius $\alpha = 0.2$.}
		\label{tab:inverse_threelayers}
		\centering

		\begin{tabular}{cccc}
			\toprule
			Medium & $\frac{|\hat{\bfx} - \hat{\bfx}_h|}{|\hat{\bfx}|}$ & $\frac{|d - d_h|}{|d|}$ & $\frac{|\alpha - \alpha_h|}{|\alpha|}$ \\
			\midrule
			Homogeneous & 3.616e-03 & 5.432e-02 & 2.845e-02 \\
		 	Heterogeneous & 1.164e-02 & 3.326e-02 & 8.000e-02 \\
			\bottomrule
		\end{tabular}
	\end{table}

	\subsubsection{A realistic head model}

	Finally, we consider a realistic head mesh. We use the Colin27 adult brain atlas (version~2, see~\cite{Colin27,Fang2010}). As shown in \autoref{fig:head_mesh}, the tetrahedrons are smaller in some regions, to reflect the complexity of the brain. The elements are of size between \num{4.67e-5} and \SI{2e-2}{\meter}, leading to a total number of $N_t = \num{425224}$ tetrahedrons and $N_e = \num{499136}$ edges. We generate a new database with this mesh in the case of one perturbation in a homogeneous background. We refer to \autoref{sec:numerical_results} for the values of the physical parameters. Data used for the inverse problem are generated with another head mesh, slightly different, of same mesh size but with $N_e = \num{535921}$ edges. The errors are reported in \autoref{tab:inverse_head}. In \autoref{fig:inverse_head}, we compare graphically the expected perturbation and the approximated one. These results show that the algorithm is not limited to the academic case of the unit ball, but can also be applied on more realistic geometries.

	\begin{table}[hbt]
		\caption{Realistic head mesh. Reconstruction of a spherical perturbation in the conductivity.}
		\label{tab:inverse_head}
		\centering

		\begin{tabular}{cccc}
			\toprule
			Domain geometry & $\frac{|\hat{\bfx} - \hat{\bfx}_h|}{|\hat{\bfx}|}$ & $\frac{|d - d_h|}{|d|}$ & $\frac{|\alpha - \alpha_h|}{|\alpha|}$ \\
			\midrule
			Head & 7.007e-03 & 2.708e-03 & 7.198e-03 \\
			\bottomrule
		\end{tabular}
	\end{table}

	\begin{figure}[hbt]
		\centering
		\includegraphics[width=0.49\textwidth]{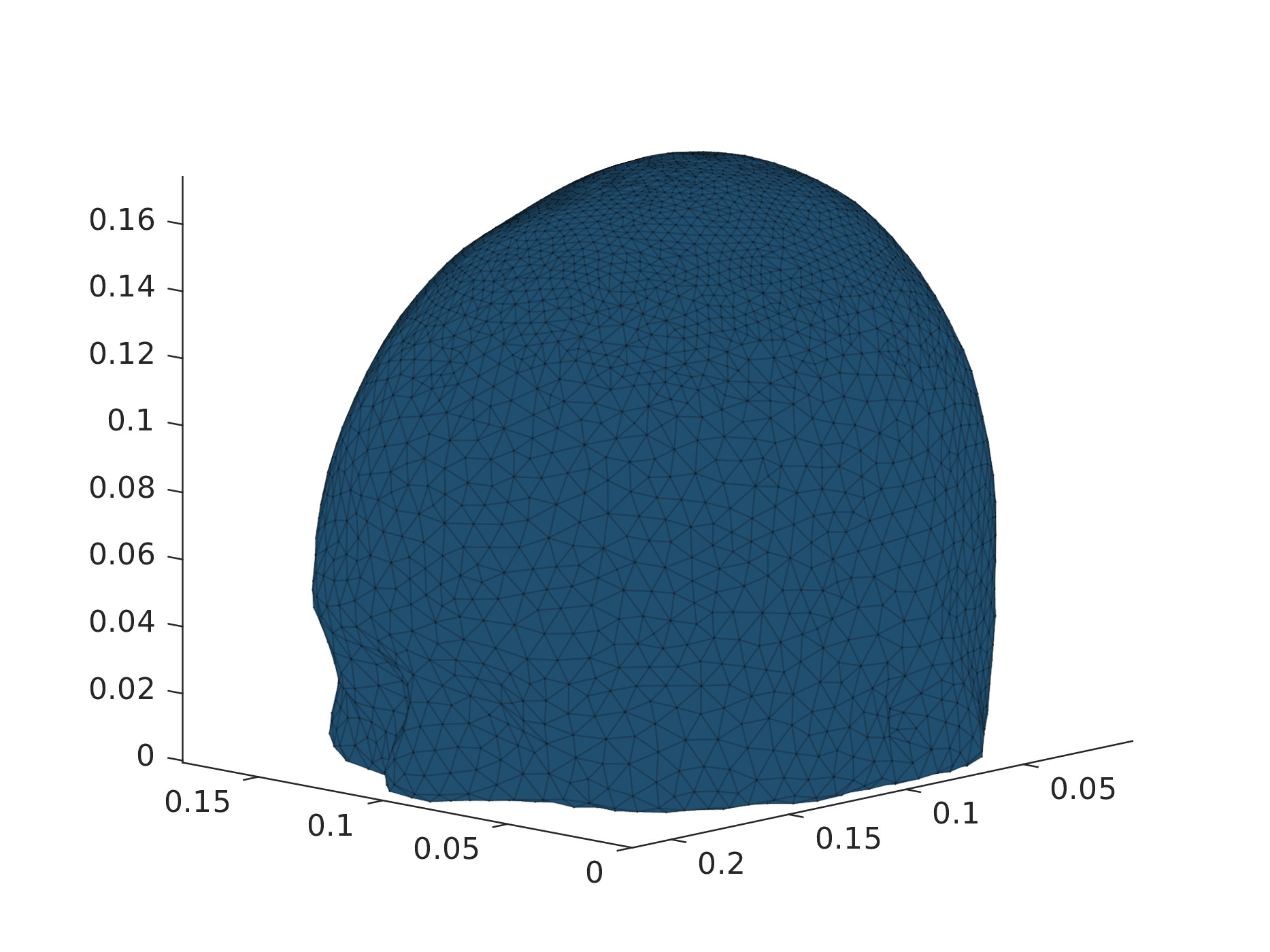}
		\hfill
		\includegraphics[width=0.49\textwidth]{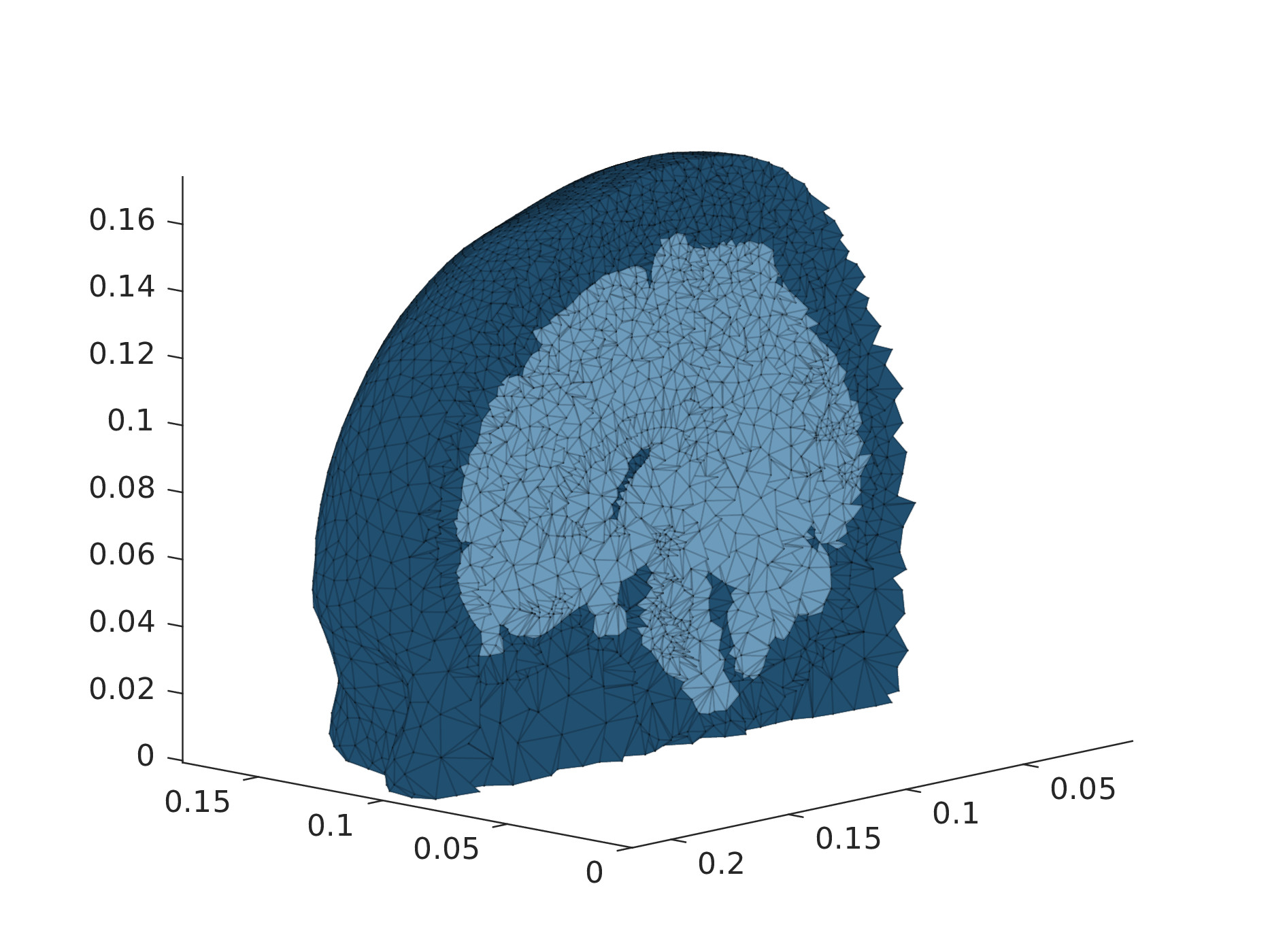}
		\caption{Realistic head model. Left: view on the boundary. Right: cut in the middle of the $x$-axis.}
		\label{fig:head_mesh}
	\end{figure}

	\begin{figure}[hbt]
		\centering
		\includegraphics[width=0.49\textwidth]{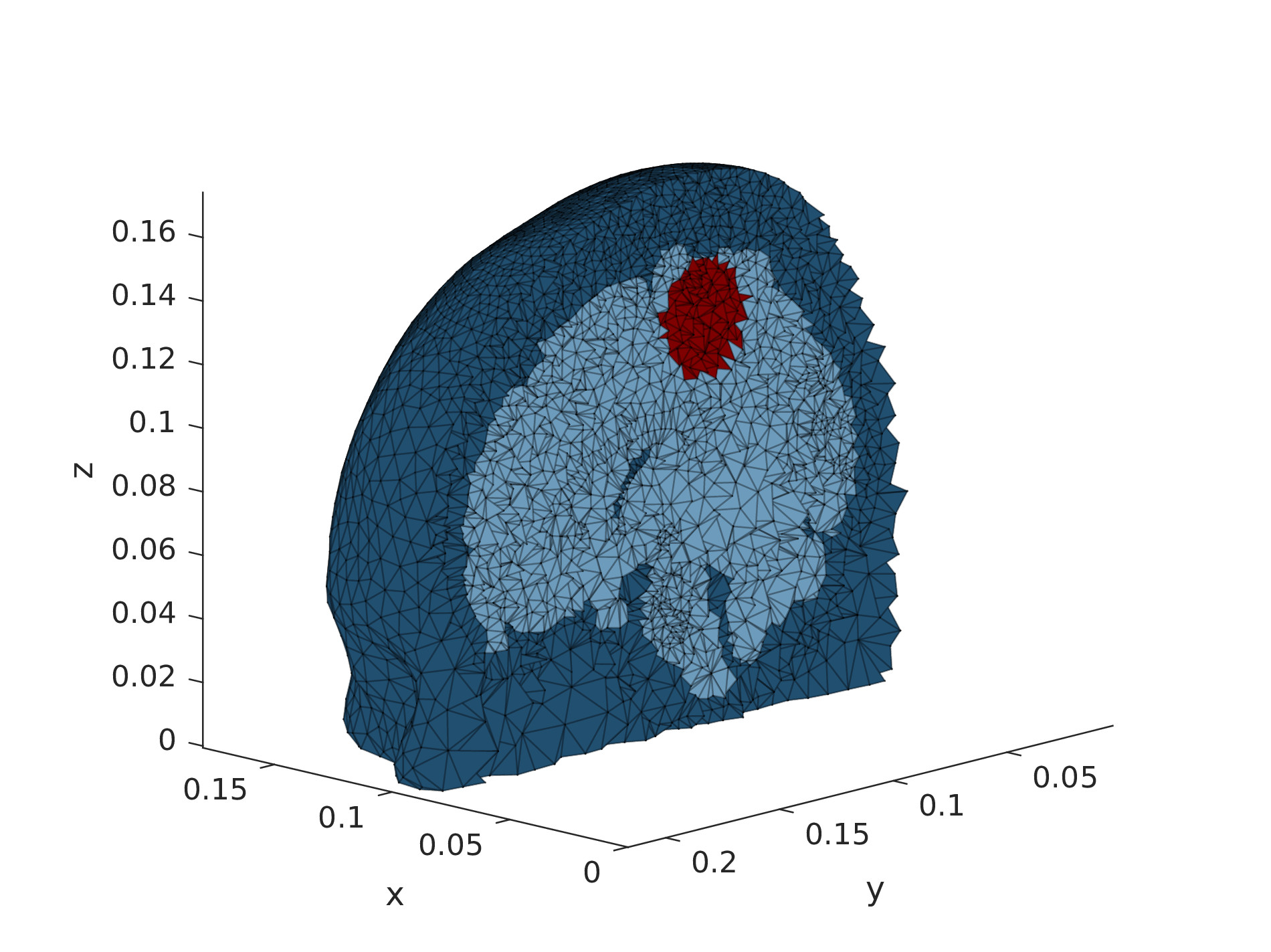}
		\hfill
		\includegraphics[width=0.49\textwidth]{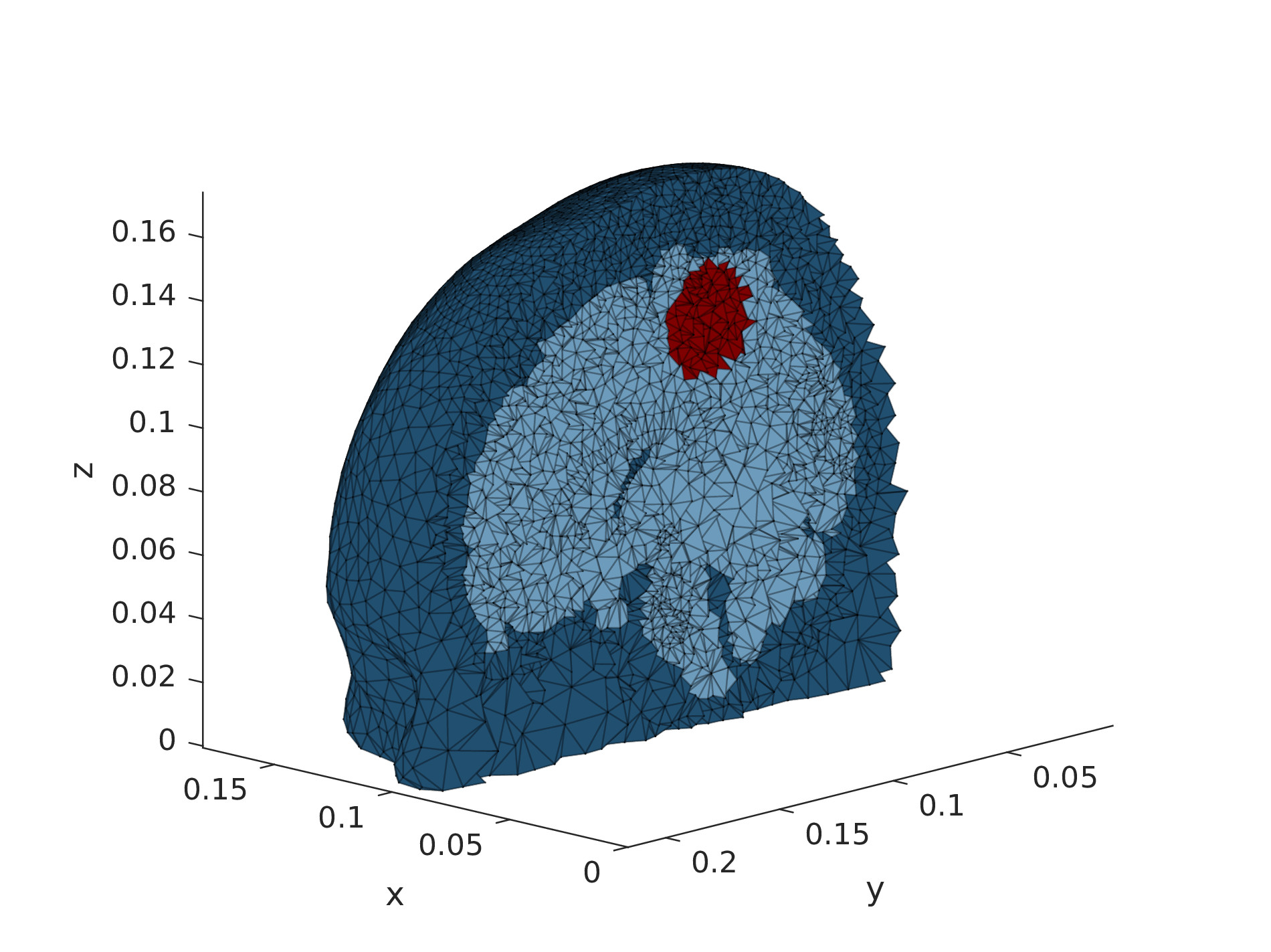}
		\caption{Realistic head mesh. Localization (in red) of a spherical perturbation in the conductivity. Left: expected perturbation. Right: result of the algorithm. The brain region, where the tetraedrons are smaller, is shown in lightblue for more readibility.}
		\label{fig:inverse_head}
	\end{figure}

	An interesting fact is that the computation time is still reasonable with this mesh, which is finer than the ball we used for the other tests. In all the tests we present here, the localization procedure is achieved in a few seconds on a personal computer (quad-core processor clocked at \SI{2.5}{\giga\hertz}, with \SI{4}{\gibi\byte} of RAM).

	\section{Conclusion and Future Works}

	In this paper, we have proposed a new and efficient algorithm for localizing small-amplitude perturbations in the electric parameters of a medium from boundary field measurements at a fixed frenquency. The approach is based on a rigourous sensitivity analysis of the electric field with respect to the variations of the permittivity and conductivity. Sensitivity is proportional to the boundary measurements of the physical field for perturbations of arbitrary shape but small amplitudes. We have proved, both theoretically and numerically, that the trace of the sensitivity of the electric field (on the boundary of the domain) contains relevant informations on the perturbations in the medium. Up to our knowledge, this is the first time that this kind of sensitivity analysis for the 3D Maxwell equations is used in the reconstruction of parameter perturbations. From an integral equation in a homogeneous background and extensive numerical simulations, we have obtained explicit relations between the sensitivity and some characteristics (center and volume) of the perturbations. These relations lead to a constructive algorithm for determining the center, the depth and the volume of inhomogeneities in the permittivity and/or the conductivity of a medium. Its implementation  makes use of diverse tools from scientific computing as, for example, 3D finite element discretization, geometric algorithms, or outlier detection for noise reduction. A large variety of three-dimensional numerical results attests the efficiency of the method: one or two perturbations with different locations and sizes, perturbation in the permittivity and/or conductivity, spherical or ellipsoidal defects, non-noisy or noisy data, a constant or piecewise constant background. The study of a general variable background would be an interesting perspective for future work. Furthermore, in view of biomedical applications (e.g. microwave imaging), we have also provided simulations in the case where the computational domain is the head. Two types of head models have been considered: the classical three-layer spherical model and a realistic one. In each configuration, the perturbations are localized with a very good accuracy, and information on its volume are obtained, too. This non-iterative localization procedure would be an interesting initial guess for gradient-based descent algorithms intended to retrieve the physical coefficient values in the perturbations. This is part of ongoing work.

	All the results have been obtained at a fixed frequency. Considering a non-ionising electromagnetic radiation, a possible application of our work is to discriminate between healthy and abnormal brain tissues. This can be achieved with a single frequency and changing the frequency does not provide a priori more information, excepted if the frequency dependance of the effective electric properties of the tissue is known. In this case, a multrifrequency approach could be interesting. Recent works on the multifrequency electrical impedance tomography (mfEIT) have been addressed~\cite{Alberti, Ammari18}.

\end{document}